\documentclass[a4paper,10pt%,draft
]{amsart}

\usepackage{amsmath,amssymb,amsthm}
\usepackage[alphabetic, abbrev]{amsrefs}
\usepackage{enumerate}
\usepackage{hyperref}
\usepackage{mathrsfs}
\usepackage{ulem}
\usepackage[all]{xy}
\newdir{ >}{{}*!/-5pt/@{>}} %% cf xyguide exercise 14
\SelectTips{cm}{} %% cf xyguide page 9

\numberwithin{equation}{section}

\newtheorem{theorem}[subsection]{Theorem}
\newtheorem{corollary}[subsection]{Corollary}
\newtheorem{lemma}[subsection]{Lemma}
\newtheorem{proposition}[subsection]{Proposition}
\theoremstyle{definition}
\newtheorem{definition}[subsection]{Definition}

\newtheorem{remark}[subsection]{Remark}

\newcommand{\bS}{\mathbb{S}}

\newcommand{\cB}{\mathcal{B}}

\newcommand{\cI}{\mathcal{I}}

\newcommand{\cL}{\mathcal{L}}

\newcommand{\cN}{\mathcal{N}}

\newcommand{\cT}{\mathcal{T}}

\DeclareMathOperator{\hocolim}{hocolim}
\DeclareMathOperator*{\hocolimsubscr}{hocolim}
\DeclareMathOperator{\holim}{holim}
\DeclareMathOperator{\colim}{colim}

\DeclareMathOperator{\sd}{sd}

\DeclareMathOperator{\THH}{THH}
\DeclareMathOperator{\TC}{TC}

\newcommand{\ot}{\leftarrow}

\newcommand{\op}{{\mathrm{op}}}
\newcommand{\id}{{\mathrm{id}}}

\newcommand{\sm}{\wedge}
\newcommand{\sma}{\wedge}
\newcommand{\barsmash}{\barwedge}
\newcommand{\iso}{\cong}
\newcommand{\concat}{\sqcup}
\newcommand{\bld}[1]{{\mathbf{#1}}}

\newcommand{\Spsym}{\mathrm{Sp}^{\Sigma}}
\newcommand{\SpO}{\mathrm{Sp}^{O}}

\newcommand{\sh}{\mathrm{sh}}
\newcommand{\R}{\mathbb{R}}
\newcommand{\Sph}{\mathbb{S}}

\newcommand{\arxivlink}[1]{\href{http://arxiv.org/abs/#1}{\texttt{arXiv:#1}}}
\usepackage[pdftex]{graphics,color} 
\usepackage{eso-pic}
\usepackage{scrtime}
\usepackage{ifdraft}
\ifdraft{%
\AddToShipoutPicture{\put(30,30){\resizebox{!}{.4cm}%
   {{{\color[gray]{0.8}\texttt{Preliminary draft, compiled \the\year-\the\month-\the\day~at \thistime}}}}}}
}{}

\ifdraft{%
  \newcommand{\authorcomment}[1]{{\marginpar{\hspace{0.2\marginparwidth}\rule{0.6\marginparwidth}{0.75mm}\hspace{0.2\marginparwidth}}\noindent\bfseries[#1]}}
}{
  \newcommand{\authorcomment}[1]{}
}

\title[Comparing cyclotomic structures on different models for \texorpdfstring{$\THH$}{THH}]{Comparing cyclotomic structures on different models for topological Hochschild homology}

\author[E. Dotto]{Emanuele Dotto} \address{Mathematical Institute, University of Bonn, % Endenicher Allee 60, 53115 Bonn
Germany
} 
\email{dotto@math.uni-bonn.de}

\author[C. Malkiewich]{Cary Malkiewich} \address{Department of Mathematics, Binghamton University, % 1409 W Green St, Urbana, IL 61801
USA
} 
\email{malkiewich@math.binghamton.edu}

\author[I. Patchkoria]{Irakli Patchkoria} \address{Department of Mathematics, University of Aberdeen, UK}
\email{irakli.patchkoria@abdn.ac.uk}

\author[S. Sagave]{Steffen Sagave} \address{IMAPP, Radboud University Nijmegen, % PO Box 9010, 6500 GL Nijmegen
The Netherlands
} \email{s.sagave@math.ru.nl}

\author[C. Woo]{Calvin Woo} \address{Department of Mathematics, Indiana University, % 831 E 3rd St, Bloomington, IN 47405
USA
}
\email{calwoo@indiana.edu}
\date{\today}

\begin{document}
\begin{abstract}
  The topological Hochschild homology $\THH(A)$ of an orthogonal ring
  spectrum $A$ can be defined by evaluating the cyclic bar
  construction on $A$ or by applying B\"okstedt's original definition
  of $\THH$ to $A$. In this paper, we construct a chain of stable
  equivalences of cyclotomic spectra comparing these two models for
  $\THH(A)$. This implies that the two versions of topological cyclic
  homology resulting from these variants of $\THH(A)$ are equivalent.
\end{abstract}
\maketitle
\section{Introduction}
Topological cyclic homology $(\TC)$ was introduced by B\"okstedt,
Hsiang and Madsen in the influential paper~\cite{BHM93}. It is the
target of a cyclotomic trace map from algebraic $K$-theory, and the
study of $\TC$ and the cyclotomic trace map has led to many successful
computations in algebraic $K$-theory (see
e.g.~\cite{HM97}). In~\cite{BHM93}, $\TC$ was defined for
so-called \textit{functors with smash products} $F$. The first step of
the definition is to form B\"okstedt's topological Hochschild
homology $\THH(F)$~\cite{Boekstedt_THH}. This is the realization of a cyclic object, which is at each simplicial level a homotopy colimit of loop spaces of smash products of the values of $F$. The $S^1$-spectrum $\THH(F)$ admits
certain ``cyclotomic'' structure maps. These give rise to a diagram of
fixed point spectra of $\THH(F)$, and $\TC(F)$ is defined to be the
homotopy limit of this diagram.

Soon after the invention of $\TC$, categories of spectra with a
structured smash product were introduced, including the categories of symmetric \cite{HSS00} and orthogonal spectra~\cite{MMSS}. In these categories, functors with smash
product can be expressed as ring spectra, which are by definition
monoids with respect to the smash product. This viewpoint allows for a
conceptually simpler definition of topological Hochschild homology:
Under a mild cofibrancy hypothesis on a ring spectrum $A$, it can be
defined to be the cyclic bar construction $B^{\mathrm{cy}}(A)$. This
is the realization of a cyclic object given in degree $q$ by $A^{\sm
  (q+1)}$, the $(q+1)$-fold power of $A$ with respect to the smash
product.

For many years, it was not known how to equip the cyclic bar
construction $B^{\mathrm{cy}}(A)$ with a cyclotomic structure so that
it can be used as a basis for the definition of $\TC(A)$. This
situation changed through progress on the understanding of
equivariant orthogonal spectra made in recent years: Based on results about norm
constructions and geometric fixed points obtained by Hill, Hopkins,
and Ravenel \cite{HHR}, Angeltveit et al.~\cite{ABG15} defined a
cyclotomic structure on the cyclic bar construction
$B^{\mathrm{cy}}(A)$ of a sufficiently cofibrant orthogonal ring
spectrum $A$.  Using a fibrant replacement functor introduced by
Blumberg and Mandell \cite{BM16}, this cyclotomic structure leads to a
definition of $\TC(A)$. Independent of this, Stolz~\cite{Stolz_equivariant} studied equivariant
orthogonal spectra by implementing a flat (or~$\mathbb S$-) model
structure. Employing a fibrant replacement functor from Kro's thesis
\cite[\S 5.1]{Kro_involutions}, Stolz' work also leads to a
construction of $\TC(A)$ based on the cyclic bar construction, see
\cite[Section 3.6.1]{Stolz_equivariant} and \cite[Section
4.6.1]{BDS16}.

The first main result of the present paper is that the cyclotomic
spectrum coming from $B^{\mathrm{cy}}(A)$ is equivalent to the cyclotomic spectrum $\THH(A)$ considered
in~\cite{BHM93}: 
\begin{theorem}\label{thm:comparison-in-introduction}
Let $A$ be a flat orthogonal ring spectrum. 
Then there is a chain of stable equivalences of cyclotomic spectra 
\[ B^{\mathrm{cy}}(A) \to  \mathrm{thh}(\sh{}A) 
\ot \THH(A)\] 
relating the cyclic bar construction and B\"okstedt's model for $\THH$.
\end{theorem}
In the theorem, being flat is a mild cofibrancy assumption on the underlying spectrum that is in particular satisfied by cofibrant associative and cofibrant commutative orthogonal ring spectra.

We will explain briefly the relationship between this comparison and earlier work. On the level of topological Hochschild homology, $B^{\mathrm{cy}}(A)$
has been compared to B\"okstedt's $\THH(A)$ by Shipley
\cite{MR1740756}, and one step in the argument has recently been
corrected in~\cite{MR3513565}. However, the intermediate steps in this
comparison do not appear to admit cyclotomic structures, and thus this
comparison does not provide an equivalence of cyclotomic spectra.  We note that as a byproduct, our theorem gives a direct comparison of the different
models for topological Hochschild homology, different from the one
carried out by Shipley~\cite{MR1740756}. The fact that we work with
orthogonal rather than symmetric spectra is no limitation since the
corresponding categories of ring spectra are Quillen
equivalent.

Angeltveit et al.~\cite{MR3556283} also give evidence that $B^{\mathrm{cy}}(A)$ and $\THH(A)$ should agree as cyclotomic spectra, by showing that the equivariant homotopy types of the individual levels of the two cyclic spectra are equivalent; however these equivalences are not compatible with the cyclic structure maps.  Our comparison theorem also gives a different proof of this main result of~\cite{MR3556283}. Finally, in the special case where $A$ is commutative, a comparison is described by Brun, Dundas, and Stolz \cite[\S 4.5.19]{BDS16} between the Loday functor modeled on smash products of orthogonal spectra and the Loday functor in $\Gamma$-spaces modeled on B\" okstedt smash products from the paper of Brun, Carlsson and Dundas \cite{BCD}. They point out that the behavior of the genuine fixed points is the same, but highlight some of the remaining work needed to make the comparison between the genuine fixed points.

Nikolaus and Scholze \cite{NS17} recently introduced a new and conceptually simpler definition of cyclotomic spectra and constructed a cyclotomic structure on topological Hochschild homology. In \cite[Theorem III.6.1 and Corollary III.6.8]{NS17}, they show that their approach provides a cyclotomic spectrum that is equivalent to the cyclotomic spectrum $\THH(A)$ considered in~\cite{BHM93}. 
However, the paper \cite{NS17} does not give a comparison to the cyclotomic spectrum $B^{\mathrm{cy}}(A)$ of \cite{ABG15}. Combining the comparison results of \cite{NS17} with our Theorem~\ref{thm:comparison-in-introduction} thus shows that the Nikolaus--Scholze model for the cyclotomic spectrum $\THH(A)$ is equivalent to the cyclotomic spectrum $B^{\mathrm{cy}}(A)$.

The proof of Theorem~\ref{thm:comparison-in-introduction} may be
summarized as follows. At simplicial level $q \geq 0$ the chain of equivalences in Theorem~\ref{thm:comparison-in-introduction} takes the form
\begin{multline} A^{\sma (q+1)}\label{eq:levelwise-zigzag-intro}
\xrightarrow{\sim} \hocolimsubscr_{(\bld{n_0},\ldots,\bld{n_q}) \in \cI^{\times (q+1)}} \Omega^{n_0+\ldots+n_q} (\sh^{n_0}A \sma \ldots \sma \sh^{n_q}A)
\\ \xleftarrow{\sim} \hocolimsubscr_{(\bld{n_0},\ldots,\bld{n_q}) \in \cI^{\times (q+1)}} \Omega^{n_0+\ldots+n_q} \Sigma^\infty (A_{n_0} \sma \ldots \sma A_{n_q}). \end{multline}
The middle term has a well-defined homotopy type, by a relatively new technical result~\cite{DMPP} that the shift functor $\sh$ for orthogonal spectra preserves the property of being flat. As $q$ varies, these equivalences fit into a zig-zag of maps of cyclic objects, which are good in the appropriate sense so that the realizations are also equivalent. For the bookkeeping of structure maps, it seems easiest to regard the middle and right-hand terms as special cases of a more general construction. It takes as input a ring object $E$ in symmetric-orthogonal bispectra, and returns as output a cyclic orthogonal spectrum $\mathrm{thh}_\bullet(E)$. We also use this general framework to define a cyclotomic structure on $|\mathrm{thh}_\bullet(E)|$ and therefore on the terms in the above zig-zag. Finally, we show that the maps of the zig-zag respect the cyclotomic structure and give equivariant equivalences on each simplicial level.

\subsection{Topological cyclic homology}
To construct the topological cyclic homology from the cyclotomic spectrum $B^{\mathrm{cy}}(A)$ for
a sufficiently cofibrant orthogonal ring spectrum $A$, it is necessary to first replace $B^{\mathrm{cy}}(A)$
by a fibrant cyclotomic spectrum $B^{\mathrm{cy}}(A)^{\mathrm{fib}}$ to ensure that the $C_{p^n}$-fixed points appearing
in the definition of $\TC$ capture a well-defined homotopy type. This can be achieved by using the fibrant replacement functor for the model${}^*$ structure on cyclotomic spectra of ~\cite{BM16}. 

\begin{theorem}\label{thm:TC-comparison-intro} Let $A$ be a flat orthogonal ring spectrum. For every prime $p$, there is a chain of maps  
\[\TC((B^{\mathrm{cy}}(A))^{\mathrm{fib}};p) \xrightarrow{\sim}  \TC((\mathrm{thh}(\sh{}A))^{\mathrm{fib}};p) \xleftarrow{\sim} \TC((\THH(A))^{\mathrm{fib}};p) \ot \TC(A;p) \]
such that the first two maps are stable equivalences. If $A$ is in addition
strictly connective, then the last map is also a stable equivalence.  

An analogous statement holds for Goodwillie's integral $\TC$.
\end{theorem}
In the theorem, the first three instances of $\TC$ take as input cyclotomic
spectra whose cyclotomic structure maps involve geometric fixed points, 
while the last instance is the classical $\TC(A;p)$ of \cite{BHM93} defined in terms of fixed
points. \textit{Strictly connective} means that $A_{n}$ is $(n-1)$-connective. This
condition implies that $\THH(A)$ is fibrant as a $C_{p^n}$-$\Omega$-spectrum and does not need to be fibrantly replaced.  

In particular, Theorem~\ref{thm:TC-comparison-intro} shows that results about topological cyclic homology verified
in one of the models carry over to the other one. Its proof is based on 
Theorem~\ref{thm:comparison-in-introduction} and a result about the compatibility of cyclotomic structures defined
in terms of geometric and categorical fixed points.

\subsection{Organization}
In Section~\ref{sec:bispectra}, we develop a framework for bispectra that is convenient for keeping track of the coherences in the zig-zag~\eqref{eq:levelwise-zigzag-intro}. This is used in Section~\ref{sec:thh} to construct the chain of equivalences in
Theorem~\ref{thm:comparison-in-introduction} at the level of orthogonal
spectra with $S^1$-action. After setting up foundations about orthogonal $G$-spectra in
Section~\ref{sec:orthogonal-G}, we show in Section~\ref{sec:cyclotomic-str} that the
spectra in this chain admit compatible cyclotomic
structures and thereby prove Theorem~\ref{thm:comparison-in-introduction}.
In the final Section~\ref{sec:TC-comparison} we show how this leads to the
comparison of topological cyclic homology spectra formulated in
Theorem~\ref{thm:TC-comparison-intro}.

\subsection{Notation and conventions}
We write $\cT$ for the category of compactly generated weak Hausdorff spaces and $\cT_*$ for the corresponding category of based spaces. We use the Bousfield--Kan formula as our model for homotopy colimits. 

\subsection{Acknowledgments}
The authors would like to thank Andrew Blumberg, Amalie H\o genhaven, Michael Mandell, Kristian Moi, Thomas Nikolaus, Stefan Schwe\-de, and Martin Stolz for helpful conversations related to this project. Moreover, the authors would like to thank the referee for a detailed report that helped to substantially improve this paper. C.\,M. was supported by an AMS Simons Travel Grant. I.\,P. was supported by the Danish National Research Foundation through the Centre for Symmetry and Deformation (DNRF92) and by the German Research Foundation Schwerpunktprogramm 1786. E.\,D. and I.\,P. were supported by the Hausdorff Center for Mathematics at the University of Bonn. C.\,M. and C.\,W. thank the Hausdorff Research Institute for Mathematics in Bonn for their hospitality while the draft of this paper was finalized. 

\section{Diagram spectra and bispectra}\label{sec:bispectra}
In this section we review basic notions about diagram spectra and introduce symmetric orthogonal bispectra that will be used later. 

\subsection{Diagram spaces and spectra} We begin by recalling
from~\cite{MMSS} how symmetric and orthogonal spectra can be
described as enriched functors.

Let $\cI$ be the category of finite sets $\bld{m} = \{1, \dots, m\}$
with $m\geq 0$ and morphisms the injective maps. Let $\cL$ be the
topological category of finite dimensional real inner product spaces
and $\mathbb R$-linear isometric embeddings. Let $\cN$ be the
category associated with the partially ordered set $(\mathbb N,\leq)$. We note that there are canonical functors $\cN \to \cI, m \mapsto \bld{m}$ and $\cI \to \cL, \bld{m}\mapsto \mathbb R^m$.

If $V$ is an object of $\cL$, we write $S^V$ for the one point
compactification of $V$. It inherits an action of the
orthogonal group $O(V) = \cL(V,V)$. If $V = \mathbb R^m$, then $S^V$
is the $m$-sphere~$S^m$. The space $\cL(V,W)$ is the base space of a
``complementary'' vector bundle with total space
$\{(w,\phi) \in W \times \cL(V,W) \, | \, w \perp \phi(V) \},$ and we
let $\cL_S(V,W)$ be the Thom space of this bundle. The spaces
$\cL_S(V,W)$ assemble to a $\cT_*$-enriched category $\cL_S$ with the same
objects as $\cL$ and composition induced by the composition of $\cL$.
(Other authors write $\mathscr J$ instead of $\cL_S$.)
The choice of a linear isometry $\phi\colon V \to W$ induces a
homeomorphism
\[ \cL_S(V,W) \iso O(W)_+ \sm_{O(W-\phi(V))} S^{W-\phi(V)} \ .\]
Following~\cite[3.1]{Schlichtkrull_Thom-symmetric}, we define $\cI_S$ to be the $\cT_*$-enriched category
with the same objects as~$\cI$ and morphism spaces
\[ \cI_S(\bld{m},\bld{n}) = \textstyle\bigvee_{\alpha\colon
  \bld{m}\to\bld{n} \in \cI} S^{\bld{n}-\alpha}\]
where $\bld{n}-\alpha$ denotes the complement of the image of $\alpha$. 
If $\alpha\colon \bld{m} \to \bld{n}$ and
${\beta \colon \bld{n} \to \bld{p}}$ are injections, then the
composition in $\cI_S$ is defined by the composition in $\cI$ and the
homeomorphism
$ S^{\bld{p}-\beta} \sm S^{\bld{n}-\alpha} \to
S^{\bld{p}-\beta\alpha}$
induced by the linear isometry
$ {\mathbb R^{\bld{p}-\beta} \oplus \mathbb R^{\bld{n}-\alpha}} \to
\mathbb R^{\bld{p}-\beta \alpha}$
determined by $\alpha$ and $\beta$. The functor
$\mathbb R^{-}\colon \cI\to \cL$ induces a functor $\cI_S \to \cL_S$
that we also denote by $\mathbb R^{-}$. On objects it sends $\bld{m}$ to $\mathbb R^m$, and $\cI_S(\bld{m},\bld{n}) \to \cL_S(\mathbb R^m,\mathbb R^n)$ is the canonical map induced by the map $\cI(\bld{m},\bld{n}) \to \cL(\mathbb R^m,\mathbb R^n)$ sending an injection to the associated isometry.

For the present paper, it is convenient to view orthogonal and symmetric
spectra as enriched functors (compare~\cite[Examples 4.2 and 4.4]{MMSS}). 
\begin{definition}
\begin{enumerate}[(i)]
\item A \textit{symmetric spectrum} is a based continuous functor
  ${\cI_S \!\to\! \cT_*}$.
\item An \textit{orthogonal spectrum} is a based continuous functor
  $\cL_S \to \cT_*$.
\end{enumerate}
\end{definition}
We write $\Spsym$ and $\SpO$ for the resulting functor categories.
Since $\cI_S$ and $\cL_S$ are symmetric monoidal under the disjoint
union and the sum, $\Spsym$ and $\SpO$ inherit $\sm$-products
defined as Day convolution products. We refer to associative (but not
necessarily commutative) monoids as symmetric or orthogonal
\textit{ring spectra} and recall that they are given by lax monoidal
functors from $\cI_S$ or $\cL_S$ to $\cT_*$. 

Let $\cT_*^{\cL}$ be the category of based continuous functors $\cL \to \cT_*$. (These are different from orthogonal spectra $\SpO$, which are based continuous functors $\cL_S \to \cT_*$.) There is a functor 
\begin{equation}\label{eq:OmegaL} \Omega^{\cL} \colon \SpO \to \cT_*^{\cL}, \quad X\mapsto \left(V\mapsto \mathrm{Map}(S^V,X(V))\right) \ .
\end{equation}
On morphism spaces $\Omega^{\cL}$  is given by the continuous maps
\[\begin{split}
&\cL(V,W) \sm \mathrm{Map}(S^V,X(V)) \to \mathrm{Map}(S^W,X(W)) \\
 &(\phi, f)\mapsto \left(S^W \xrightarrow{\iso} S^V \sm S^{W-\phi(V)} \xrightarrow{f \sm \mathrm{id}} X(V) \sm S^{W-\phi(V)}\xrightarrow{\phi_*} X(W)\right).
\end{split}
\] 
For symmetric spectra there is an analogously defined functor 
\begin{equation}\label{eq:OmegaI} \Omega^{\cI} \colon \Spsym \to \cT_*^{\cI}, \quad X\mapsto \left(\bld{n}\mapsto \mathrm{Map}(S^n,X_n)\right) \ .
\end{equation}
If $X \colon \cL_S \to \cT_*$ is an orthogonal spectrum and $m \geq 0$
is an integer, then the $m$-fold shift of $X$ is the orthogonal
spectrum $\sh^m X = X(\mathbb R^m \oplus -) \colon \cL_S \to \cT_*$.
For later use, we note that the restrictions of the $\cL$-diagrams
$\Omega^{\cL}(\sh^mX)$ to $\cN$-diagrams along the canonical
functor $\cN \to \cL$ can be used to detect the homotopy groups and therefore the $\pi_*$-isomorphisms between orthogonal spectra:
\begin{lemma}\label{lem:pi-star-orthogonal-detection}
  A map of orthogonal spectra $X \to Y$ is a $\pi_*$-isomorphism if
  and only if the induced map
  $\hocolim_{\cN}\Omega^{\cL}(\sh^mX) \to
  \hocolim_{\cN}\Omega^{\cL}(\sh^mY)$
  is a weak homotopy equivalence of based spaces for all $m\geq
  0$. \qed
\end{lemma}
An analogous statement holds for $\pi_*$-isomorphisms of symmetric
spectra.
\begin{definition}
  A map of orthogonal spectra $X \to Y$ is a \textit{flat cofibration} if it is the retract of
  a relative cell complex built from cells of the form
  $ \{ \cL_S(V,-) \sm_{O(V)}(i \times O(V)/H)_+ \}$ where $V$ is an
  object of $\cL$, $H$ is a closed subgroup of $O(V)$, and
  $i \colon \partial D^k \to D^k$, $k \geq 0$, is a generating
  cofibration for the standard model structure on $\cT$. An orthogonal spectrum 
  $X$ is flat if $* \to X$ is a flat cofibration, and an orthogonal
  ring spectrum is flat if its underlying orthogonal spectrum is.
\end{definition}
\begin{remark}
  This use of the term flat follows Schwede's
  terminology~\cite{Schwede_global}. Our flat orthogonal spectra are
  called \textit{$\bS$-cofibrant} in~\cite{Stolz_equivariant} and
  should not be confused with the more general flat objects of~\cite[Definition B.15]{HHR}.
\end{remark}
Flat cofibrations of orthogonal spectra are the cofibrations in a stable model
structure on $\SpO$ whose weak equivalences are the $\pi_*$-isomorphisms~\cite[Proposition
1.3.10]{Stolz_equivariant}.
By~\cite[Theorem 4.1(3)]{SS02}, both the projective model structure \cite{MM02} and the flat model structure \cite{Stolz_equivariant} on $\SpO$ lift to a model structure on the category of associative orthogonal ring spectra $\mathcal{A}\SpO$. Analogously, the category of commutative orthogonal ring spectra $\mathcal{C}\SpO$ admits a positive projective \cite[Theorem 15.1]{MMSS} and a positive flat \cite[Theorem 1.3.28]{Stolz_equivariant} model structure.
\begin{lemma}\label{lem:unit-is-flat-cof}
Let $A$ be an orthogonal ring spectrum.
\begin{enumerate}[(i)]
\item Suppose that $A$ is not the terminal ring spectrum. Then the unit $\bS \to A$ is a flat cofibration if and only if $A$ is flat as an orthogonal spectrum. \item If $A$ is cofibrant in one of the above model structures on $\mathcal{A}\SpO$ or $\mathcal{C}\SpO$ then the unit 
$\bS \to A$ is a flat cofibration. 
\end{enumerate}
\end{lemma}
\begin{proof}
For the nontrivial implication of (i) we assume that $A$ is flat and non-zero. Then $S^0 \to A_0$ is injective. We write $A$ as the retract of a cell complex $B$. Since attaching cells of the form $\cL_S(\R^m,-) \sma i_+$ with $m > 0$ doesn't change the spectral degree zero part, we can pick the first cell of $B$ of the form $\cL_S(\R^0,-) \sma i_+$ with $i\colon \partial D^n \to D^n$ a generating cofibration that accommodates the image of the non-basepoint of $S^0$ in its image. After possibly subdividing $D^n$, we may assume that the non-basepoint of $S^0$ is a zero-cell $B_0$. Re-indexing the cells of $B$ so that this new $0$-cell comes first, we see that $\bS \to B$ and thus $\bS \to A$ is a retract of a relative cell complex.

For (ii), we only need to address the flat model structures since projective cofibrant implies flat cofibrant. The associative case follows from~\cite[Theorem 4.1(3)]{SS02} and the commutative case follows from \cite[Theorem 1.3.30]{Stolz_equivariant}.
\end{proof}
\begin{lemma}\phantomsection{}\label{lem:flat-preservation}\begin{enumerate}[(i)]
\item Smashing with flat orthogonal spectra preservers flat cofibrations. In particular, if $X$ and $Y$ are flat in $\SpO$, then so is $X \sm Y$. 
\item If $X$ is flat in $\SpO$, then $X \sm -$ preserves $\pi_*$-isomorphisms. 
\end{enumerate}
\end{lemma}
\begin{proof}
  Part (i) holds because the model structure of~\cite[Proposition
  1.3.10]{Stolz_equivariant} is monoidal, and part (ii)
  is~\cite[Proposition 1.3.11]{Stolz_equivariant}.
\end{proof}
\begin{remark}
  The flat model structure on orthogonal spectra and the previous
  lemma also appear in more general equivariant contexts
  in~\cite[Proposition 2.10.1]{BDS16} and~\cite[Theorem III.5.10]{Schwede_global}.
\end{remark}

The following result about shifts from \cite[Appendix A.2]{DMPP} is a key technical ingredient for the present paper.
\begin{theorem}\label{thm:shifts-of-flat}
  If $X$ is a flat orthogonal spectrum, then its $m$-fold shift $\sh^m
  X$ is also flat.
\end{theorem}

\subsection{Symmetric orthogonal bispectra}
The following notion of bispectra will be convenient for keeping track of coherences.
\begin{definition}\label{def:bispectra}
  A symmetric spectrum object in orthogonal spectra is a based continuous functor 
  $E\colon \cI_S \sm \cL_S \to \cT_*$. We simply call these objects \textit{bispectra} 
  and write $E^n = E(n,-)$ for the orthogonal spectrum obtained by fixing the
  symmetric spectrum degree of $E$.
\end{definition}
Unraveling definitions, $E$ consists of a sequence of orthogonal
spectra $E^n, n\geq 0,$ with $\Sigma_n$-actions through maps in
$\SpO{}$, and structure maps $E^n \sm S^1 \to E^{n+1}$ in $\SpO{}$ that
are compatible with the $\Sigma_n$-action in the same sense required
for ordinary symmetric spectra. So $E$ can be viewed as a symmetric
spectrum object in $\SpO{}$. We will also sometimes view $E$ as an
orthogonal spectrum object in $\Spsym{}$.

The symmetric monoidal structures on $\cI_S$ and $\cL_S$ provide a
symmetric monoidal structure on $ \cI_S \sm \cL_S$ that induces a Day
convolution product on $\Spsym(\SpO)$. Almost immediately from its definition, this product is canonically isomorphic to the smash product of symmetric spectrum objects in orthogonal spectra. We therefore call the associative monoids
in $\Spsym(\SpO)$ symmetric ring spectra of orthogonal spectra, or simply ring bispectra. They are the lax monoidal functors
$\cI_S \sm \cL_S \to \cT_*$.

We shall be interested in two types of examples for symmetric spectra in $\SpO{}$. 

\begin{definition}
  Let $X$ be an orthogonal spectrum. Then $\Sigma^{\infty}X$ is the bispectrum 
  given in level $(m,V)$ by
  $(\Sigma^{\infty}X)(m,V) = X(\mathbb R^m) \sm S^V$. The structure
  maps arise by viewing $\Sigma^{\infty}X$ as an enriched 
  functor
  \[\Sigma^{\infty}X = X(\mathbb R^{-}) \sm \cL_S(0,-) \colon \cI_S \sm
  \cL_S \to \cT_*\ .\]
\end{definition}
So $(\Sigma^{\infty}X)^n$ is the
orthogonal suspension spectrum of the $n$-th level $X_n=X(\mathbb
R^n)$ of $X$. When $A$ is an orthogonal ring spectrum, then
$\Sigma^{\infty}A$ is a smash product of two diagrams that define symmetric and orthogonal ring spectra, respectively, and is therefore a ring bispectrum.

\begin{definition}
  Let $X$ be an orthogonal spectrum. Then $\sh{}X$ is the
  bispectrum given in level $(m,V)$ by
  $(\sh{}X)(m,V) = X(\mathbb R^m \oplus V)$. The structure maps
  arise by viewing $\sh{}X$ as an enriched
  functor
  \[\sh{}X = X(\mathbb R^{-}\oplus -) \colon \cI_S \sm \cL_S \to
  \cT_*\ .\]
\end{definition}
For each value of $n \geq 0$, the orthogonal spectrum $\sh^nX = (\sh{} X)^n =
(\sh{}X)(n,-)$ is the $n$-fold shift of $X$. Since the composite
\[ \cI_S \sm \cL_S \xrightarrow{\mathbb R^{-} \sm \mathrm{id}} \cL_S \sm \cL_S \xrightarrow{\oplus} \cL_S \]
is a strong symmetric monoidal functor, it follows that an orthogonal ring spectrum $A$ gives rise to a ring bispectrum $\sh{}A$.

We notice that the structure maps of $X$ induce a canonical morphism of bispectra
\begin{equation}\label{eq:suspension-to-shift}
\Sigma^{\infty}X \to \sh{}X 
\end{equation} 
that is a morphism of ring bispectra when $X$ is an orthogonal ring spectrum.

\begin{lemma}\label{lem:suspension-shift-equivalence}
Let $X$ be an orthogonal spectrum. In each orthogonal level $V$, the map~\eqref{eq:suspension-to-shift} is a $\pi_*$-isomorphism of symmetric spectra $(\Sigma^{\infty}X)(-,V) \to (\sh{}X)(-,V)$. 
\end{lemma}
\begin{proof}
The map in question is the canonical map $X(\mathbb R^{-}) \sm S^V \to X(\mathbb R^{-} \oplus V)$. Essentially by the definition of semistability, this map is a $\pi_*$-isomorphism if and only if the symmetric spectrum $X(\mathbb R^{-})$ is semistable; see~\cite[Proposition 5.6.2(2)]{HSS00}. The claim follows because $X(\mathbb R^{-})$ is the underlying symmetric spectrum of an orthogonal spectrum. 
\end{proof}
The following definition is motivated by Lemma~\ref{lem:pi-star-orthogonal-detection} and provides one of various equivalent ways to define $\pi_*$-isomorphisms of bispectra. To state it, we note that there is
a functor 
\[ \Omega^{\cI \times \cL} \colon \Spsym(\SpO) \to \cT_*^{\cI\times \cL}, \quad E\mapsto ((\bld{m},V) \mapsto \mathrm{Map}(S^m \sm S^V,E^m(V)))\]
that is defined on morphisms in a manner similar to $\Omega^{\cL}$ and $\Omega^{\cI}$. An $\cI\times \cL$-diagram may be viewed as an $\cN\times \cN$-diagram by restricting along the inclusion $\cN\times \cN \to \cI\times \cL$. 
\begin{definition}
A morphism of bispectra $D \to E$ is a \textit{$\pi_*$-isomorphism} if for all $k,l \geq 0$ the induced map 
\[ \hocolim_{\cN\times\cN} \Omega^{\cI\times \cL}(D(\bld{k}\concat -, \mathbb R^l \oplus -)) \to \hocolim_{\cN\times\cN} \Omega^{\cI\times \cL}(E(\bld{k}\concat -, \mathbb R^l \oplus -))\]
is a weak homotopy equivalence of spaces. 
\end{definition}

\begin{corollary}\label{cor:shiftandsusp}
If $X$ is an orthogonal spectrum, then $\Sigma^{\infty}X \to \sh{}X$ is a $\pi_*$-isomorphism. 
\end{corollary}
\begin{proof}
  Lemma~\ref{lem:suspension-shift-equivalence} implies that for fixed $V$ and $k,l\geq 0$, the natural transformation
  \[ \mathrm{Map}(S^{-}\sm S^V, (\Sigma^{\infty}X)(\bld{k}\concat -,\mathbb R^l \oplus V)) \to \mathrm{Map}(S^{-}\sm S^V, (\sh{}X)(\bld{k}\concat -,\mathbb R^l \oplus V)) \]
  induces a weak equivalence when evaluating the homotopy colimit over $\cN$. The claim follows by computing
  $\hocolim_{\cN\times\cN}$ as an iterated homotopy colimit.
\end{proof}

An orthogonal spectrum defines an endofunctor $X \sm - \colon \Spsym(\SpO) \to \Spsym(\SpO)$ by forming the smash product of orthogonal spectra $X \sma E^n$ for each level $n$.
\begin{lemma}\label{lem:smashing-with-flat-preserves-diagonal-isos}
  Let $D \to E$ be a $\pi_*$-isomorphism of bispectra and let $X$ be a
  flat orthogonal spectrum. Then $X \sm D \to X \sm E$ is a
  $\pi_*$-isomorphism of bispectra. 
\end{lemma}
\begin{proof}
  We argue with stabilization $\mathrm{Sp}^{\mathbb N}(\SpO)$ of the
  flat stable model structure on $\SpO$ (compare e.g.~\cite[Section
  3]{Hovey-general}).  One can check that a map in $\Spsym(\SpO)$ is a
  $\pi_*$-isomorphism if and only if its underlying map in
  $\mathrm{Sp}^{\mathbb N}(\SpO)$ is a weak equivalence. Moreover,
  $\mathrm{Sp}^{\mathbb N}(\SpO)$ is also tensored over $\SpO$ and the
  tensor commutes with the forgetful functor from
  $\Spsym(\SpO)$. Since a cofibrant replacement $E^c \to E$ in
  $\mathrm{Sp}^{\mathbb N}(\SpO)$ can be chosen to be a weak
  equivalence of spaces in every bi-degree, the fact that $X\sm -$
  sends level equivalences of orthogonal spectra to stable equivalences
  implies that $X \sm E^c \to X \sm E$ is a $\pi_*$-isomorphism. So it
  is enough to show that $X \sm -$ preserves weak equivalences between
  cofibrant objects in $\mathrm{Sp}^{\mathbb N}(\SpO)$, and this
  follows from \cite[Theorem 6.3]{Hovey-general}. \end{proof}

\section{Topological Hochschild homology}\label{sec:thh}
We recall the cyclic bar construction in orthogonal spectra. If $A$ is an orthogonal ring spectrum, we define a cyclic object $B^{\mathrm{cy}}_{\bullet}(A)$ by
\[ B^{\mathrm{cy}}_q(A) = A^{\sma (q+1)} \]
where the $\sma$ denotes the smash product of orthogonal spectra. We imagine these are arranged in a circle. Then the face maps multiply adjacent copies of $A$, the degeneracy maps insert copies of $A$ along the unit $\Sph \to A$, and the cyclic structure maps act by cyclic permutations (see e.g. \cite[Definition 4.1.2]{MR1740756}, \cite[\S 2.1]{madsen_survey}, \cite[Section 4.2]{Mal17}).

To define topological Hochschild homology we emulate this cyclic bar construction but with B\" okstedt's variant of the smash product. Specifically, let $X_0, \ldots, X_{q}$ be any $(q+1)$-tuple of bispectra in the sense of Definition~\ref{def:bispectra}. 
Their external smash product $X_0 \barsmash \ldots \barsmash X_q$ is defined to be a continuous functor $\cI_S^{\sma (q+1)} \to \SpO$ sending 
\[(\bld{n_0},\ldots,\bld{n_q})\longmapsto X_0^{n_0} \sma \ldots \sma X_q^{n_q},\]
where the superscripts are symmetric spectrum levels and $\sma$ is the smash product of orthogonal spectra. Here $\cI_S^{\sma (q+1)}$ denotes the $q+1$-fold product $\cI_S \sma \dots \sma \cI_S$ of the enriched category $\cI_S$, and we view $X_0 \barsmash \ldots \barsmash X_q$  as a $(q+1)$-multisymmetric spectrum of orthogonal spectra. By definition, the smash product of bispectra $X_0 \sma \ldots \sma X_q$ is obtained from this external smash product $X_0 \barsmash \ldots \barsmash X_q$ by left Kan extension along a direct sum functor $\oplus\colon \cI_S^{\sm (q+1)} \to \cI_S$. 

The functor $\Omega^{\cI}\colon \Spsym(\SpO) \to (\SpO)^{\cI}$ from~\eqref{eq:OmegaI} generalizes to a functor
\[ \Omega^{\cI^{\times (q+1)}}\colon (\Spsym)^{(q+1)}(\SpO) \to (\SpO)^{\cI^{\times (q+1)}} \]
from $(q+1)$-multisymmetric spectra of orthogonal spectra to $\cI^{\times (q+1)}$-diagrams of orthogonal spectra (where $\cI^{\times (q+1)}$ denotes the $(q+1)$-fold product $\cI \times \dots \times\cI$). Evaluated at the external smash product it is given by the formula
\[ \Omega^{\cI^{\times (q+1)}}(X_0 \barsmash \ldots \barsmash X_q) = \left((\bld{n_0},\dots, \bld{n_q})
  \mapsto \mathrm{Map}(S^{n_0}\sm \dots \sm S^{n_q},X_0^{n_0}\sm \dots \sm X_q^{n_q}) \right)\]
as a functor $\cI^{\times (q+1)} \to \SpO$. We define the B\" okstedt smash product of the bispectra $X_0,\ldots,X_q$ to be its homotopy colimit in the category of orthogonal spectra:
\[ \cB(X_0,\ldots,X_q) := \hocolimsubscr_{(\bld{n_0},\ldots,\bld{n_q}) \in \cI^{\times (q+1)}} \mathrm{Map}(S^{n_0}\sm \dots \sm S^{n_q},X_0^{n_0}\sm \dots \sm X_q^{n_q}). \]

We recall and slightly generalize the following functoriality property for the B\"okstedt smash product. Suppose that $X_0,\ldots,X_q$, and $Y_0,\ldots,Y_p$ are bispectra. To each map of sets $f\colon \{0,\ldots,q\} \to \{0,\ldots,p\}$, with a choice of total ordering on the preimage of each point, we define a functor 
\begin{equation}\label{ordered_loday}
f_* \colon
  \cI^{\times (q+1)} \to \cI^{\times (p+1)},\quad  (\bld{n_0},\dots, \bld{n_q})
  \mapsto \left(\textstyle{}\bigsqcup_{i \in f^{-1}(0)} \bld{n_i},\dots, \bigsqcup_{i \in f^{-1}(p)} \bld{n_i}\right)
\end{equation}
where the terms of the coproduct are ordered according to the chosen ordering of each $f^{-1}(j)$. For each such $f$, and each collection of maps of bispectra
\[ \varphi_j\colon \bigwedge_{i \in f^{-1}(j)} X_i \to Y_j, \]
we define a map of diagrams
\[ \Omega^{\cI^{\times (q+1)}}(X_0 \barsmash \ldots \barsmash X_q) \to \Omega^{\cI^{\times (p+1)}}(Y_0 \barsmash \ldots \barsmash Y_p) \circ f_* \]
using the $\varphi_j$ and the evident homeomorphism $S^{\sum_i n_i} \iso S^{\sum_j \sum_{i \in f^{-1}(j)} n_i}$. In total, the data of $f$, the orderings, and the $\{\varphi_j\}$ define a map of B\"okstedt smash products
\[ \cB(f,\{\varphi_j\})\colon \cB(X_0,\ldots,X_q) \to \cB(Y_0,\ldots,Y_p).\]
The verification of the following is straightforward and identical to the corresponding verification for symmetric spectra objects of spaces.
\begin{proposition}\label{bokstedt_functoriality}
The construction $\cB(f,\{\varphi_j\})$ respects identity maps and compositions, using the evident dictionary order to compose total orderings. \qed
\end{proposition}

We specialize to the case where all of the $X_i$ and $Y_j$ are the same bispectrum $E$. In this case we use the notations $\Omega^{\cI}_{q+1}(E)  = \Omega^{\cI^{\times (q+1)}}(E \barsmash \ldots \barsmash E)$ and define
\begin{equation}\label{eq:Omega-Iq+1}
  \mathrm{thh}_q(E)= \hocolimsubscr_{\cI^{\times (q+1)}} \Omega^{\cI}_{q+1}(E)
   = \hocolimsubscr_{(\bld{n_0},\dots, \bld{n_q})} \mathrm{Map}(S^{n_0}\sm \dots \sm S^{n_q},E^{n_0}\sm \dots \sm E^{n_q}) \ .
\end{equation}
The notation is meant to suggest that this is the $q$th simplicial
level of the simplicial object which realizes to the B\"okstedt topological
Hochschild homology of $E$.
Since $\mathrm{thh}_q(E)$ is a B\"okstedt smash product $\cB(E,\ldots,E)$, the previous proposition immediately implies the following.
\begin{corollary}\label{cyclic_structure}
If $E$ is a ring bispectrum, the multiplication and unit maps of $E$ define a cyclic structure on the orthogonal spectra $\mathrm{thh}_q(E)$.\qed
\end{corollary}
In particular, each face map concatenates two of the sets $\bld{n_i}$ and applies a multiplication map of orthogonal spectra $E^{n_i} \sma E^{n_{i+1}} \to E^{n_i + n_{i+1}}$ (with the convention that $n_{q+1}=n_0$), while each degeneracy map inserts a $\bld{0}$ and a unit map of orthogonal spectra $\Sph \to E^0$.

\begin{definition}
We write
$\mathrm{thh}(E)$ for the realization of the cyclic object $\mathrm{thh}_\bullet(E)$. In the
case where $E = \Sigma^{\infty}A$ for an orthogonal ring spectrum $A$,
we write $\THH_{\bullet}(A)$ for
$\mathrm{thh}_{\bullet}(\Sigma^{\infty}(A))$ and $\THH(A)$ for its
realization.
\end{definition}
We note that $\THH(A)$ is the orthogonal spectrum version of
B\"okstedt's original definition of topological Hochschild homology
(compare~\cite[Section 4.2]{MR1740756}). As discussed in the introduction, we are primarily interested in this case and in the case of $E = \sh A$.

Of course, since the definition of
$\Omega^{\cI}_{q+1}(E)$ involves smash products of the orthogonal
spectra $E^n$, it will only capture a well-defined homotopy type if
the $E^n$ are flat.
We will now argue the same for the realization.
We say that a map of orthogonal spectra $X \to Y$ is a levelwise $h$-cofibration if each map $X(\R^n) \to Y(\R^n)$ has the unbased homotopy extension property, and that a simplicial orthogonal spectrum is ``good'' if each degeneracy map is a levelwise $h$-cofibration. It is well-known that maps of good simplicial spectra $X_\bullet \overset\sim\to Y_\bullet$ that are stable equivalences on each simplicial level give stable equivalences on their realizations $|X_\bullet| \overset\sim\to |Y_\bullet|$.
\begin{proposition}\label{goodness}
If $E$ is a ring bispectrum and $E^n$ is flat for all $n \geq 0$, then $\mathrm{thh}_\bullet(E)$ is good.
\end{proposition}

\begin{proof}
Our degeneracy map is a map of based homotopy colimits, induced by the map of categories $f_\ast\colon \cI^{\times q} \to \cI^{\times (q+1)}$ that inserts a $\bld{0}$ and a map of diagrams
\[ \Omega^{\cI^{\times q}}(E \barsmash \ldots \barsmash E) \to \Omega^{\cI^{\times (q+1)}}(E \barsmash \ldots \barsmash E) \circ f_\ast \]
whose evaluation at an object $(\bld{n_0}, \dots, \bld{n_{q-1}})$ is the map 
\begin{equation}\label{eq:goodness_proof}
\Omega^{n_0+\ldots+n_{q-1}} \left( (E^{n_0} \!\sma\! \ldots \!\sma\! \Sph \!\sma\! \ldots \!\sma\! E^{n_{q-1}}) \to  (E^{n_0} \!\sma\! \ldots \!\sma\! E^0 \!\sma\! \ldots \!\sma\! E^{n_{q-1}}) \right)
\end{equation}
induced by identity maps $E^{n_i} \to E^{n_i}$ and the unit $\Sph \to E^0$.
Since the map $f_\ast$ is injective, it follows from the argument
given in the proof of \cite[Lemma 6.17]{LRRV17} that the map of homotopy
colimits is an $h$-cofibration as soon as for each object $(\bld{n_0},\ldots,
\bld{n_{q-1}})$ the map \eqref{eq:goodness_proof} is a levelwise 
$h$-cofibration, and for each object $(\bld{n_0},\ldots,\bld{n_q})$ the map $* 
\to \Omega^{n_0+\ldots+n_q} (E^{n_0} \sma \ldots \sma E^{n_q})$ is a levelwise $h$-cofibration, i.e., the levels are well-based.

The map $\Sph \to E^0$ is a flat cofibration by Lemma~\ref{lem:unit-is-flat-cof}(i). Therefore, before taking loops, the map \eqref{eq:goodness_proof} is a flat cofibration by Lemma~\ref{lem:flat-preservation}. Both source and target are flat and are therefore levelwise well-based. To conclude that after loops both spectra are well-based and the map is still a levelwise $h$-cofibration, it is sufficient to know that the functor $\Omega(-)$ on based spaces preserves $h$-cofibrations between well-based spaces. For closed $h$-cofibrations in the category of all topological spaces, this is shown in~\cite[Application (vi)]{solomon}. The desired statement in the category of compactly generated weak Hausdorff spaces  $\cT$ we are working in follows because $h$-cofibrations in $\cT$ are automatically closed (see e.g.~\cite[Proposition A.31]{Schwede_global}). 
\end{proof}

Returning to the B\"okstedt smash product of bispectra $X_0,\ldots,X_{q}$, the inclusion of $(\bld{0},\dots, \bld{0})$ in $\cI^{\times (q+1)}$
gives rise to a functor from the terminal category to
$\cI^{\times (q+1)}$, inducing a map of orthogonal spectra
\begin{equation} X_0^0 \sma \ldots \sma X_q^0 \to \cB(X_0,\ldots,X_q).
\end{equation}
It is immediate that the map on the right induced by $(f,\{\varphi_j\})$ agrees with the smash product of the maps $\varphi_j^0$ on the left.
Specializing again to $X_i = E$, we get maps for each $q$
\begin{equation}\label{eq:E-zero-q+1-to-thh-q-E} (E^0)^{\sma (q+1)} \to\mathrm{thh}_q(E)
\end{equation}
which induce a map of cyclic objects
\begin{equation}\label{eq:thh-sh-E-to-By-E-zero}
B^{\mathrm{cy}}_{\bullet}(E^0) \to \mathrm{thh}_\bullet(E) \ .
\end{equation}

If $A$ is an orthogonal ring spectrum, the above discussion now gives us maps of cyclic orthogonal spectra
\begin{equation}\label{eq:THH-model-comparison} 
  B^{\mathrm{cy}}_{\bullet}(A) \to  \mathrm{thh}_{\bullet}(\sh{}A) 
  \ot\mathrm{thh}_{\bullet}(\Sigma^{\infty}A) = \THH_{\bullet}(A). 
\end{equation}
The left-hand map is~\eqref{eq:thh-sh-E-to-By-E-zero} for $E = \sh{}A$, and the right-hand map arises from the morphism of ring bispectra 
$\Sigma^{\infty}A \to \sh{}A$
from~\eqref{eq:suspension-to-shift} and the naturality of the cyclic structure of Corollary~\ref{cyclic_structure}.

\begin{theorem}\label{thm:THH-model-comparison}
 Let $A$ be an orthogonal ring spectrum whose underlying orthogonal spectrum is flat. Then on geometric realizations, the chain~\eqref{eq:THH-model-comparison}
  induces the following chain of stable equivalences of orthogonal
  spectra with $S^1$-action:
\[ B^{\mathrm{cy}}(A) \to  \mathrm{thh}(\sh{}A) 
\ot\mathrm{thh}(\Sigma^{\infty}A) = \THH(A).\]
\end{theorem}
By Lemma~\ref{lem:unit-is-flat-cof}(ii), the Theorem applies to cofibrant associative and cofibrant commutative orthogonal ring spectra. We now set up some auxiliary results and prove the theorem
at the end of the section.

\begin{lemma}\label{lem:Bcy-to-thh-sh-R}
  Let $X$ be a flat orthogonal spectrum and let $q\geq 0$. Then the
  map $X^{\sm (q+1)} \to \mathrm{thh}_q(\sh{}X)$ obtained by
setting $E = \sh{}X$ in~\eqref{eq:E-zero-q+1-to-thh-q-E} is a stable equivalence. 
\end{lemma}
\begin{proof}
  Since $\cI$ has an initial object, the classifying space of
  $\cI^{\times (q+1)}$ is contractible. Therefore it is sufficient to show that
  with $E = \sh{}X$, every map in the
  $\cI^{\times (q+1)}$-diagram~\eqref{eq:Omega-Iq+1} is a stable equivalence of
  orthogonal spectra.

  When $q=0$, the map induced by $\alpha\colon \bld{m} \to \bld{n}$
  can be identified with the $m$-fold loop of the map
  $X(\mathbb R^m\oplus-) \to \Omega^{\bld{n}-\alpha(\bld{m})}(X(\mathbb
  R^n\oplus-))$
  that is adjoint to the canonical map
  \[ \sh^m X \sm S^{\bld{n}-\alpha(\bld{m})} = X(\mathbb R^m\oplus-) \sm
  S^{\bld{n}-\alpha(\bld{m})} \to X(\mathbb R^n\oplus-) = \sh^nX\ .\]
  That latter map is a $\pi_*$-isomorphism since the underlying
  symmetric spectrum of $X$ is semistable. This shows the $q=0$ case.

  Now let $q=1$. Then the map~\eqref{eq:suspension-to-shift} induces a commutative square
  \[ \xymatrix@-1pc{ (X \sm S^m) \sm (X \sm S^n) \ar[r] \ar[d] & 
    \sh^m X \sm \sh^n X \ar[d] \\
    (X \sm X) \sm S^{m+n} \ar[r] & \sh^{m+n}(X \sm X) \rlap{\ .} }\]
  By Theorem~\ref{thm:shifts-of-flat}, the orthogonal spectra
  $\sh^m X$, $\sh^m X$ and $\sh^{m+n} X$ are
  flat. Since the maps from the suspensions to the shifts are stable
  equivalences by the argument for $q=0$, this implies that the
  horizontal maps are stable equivalences. The left hand vertical map
  is an isomorphism. So we deduce that the right hand vertical map is
  a stable equivalence. This reduces the claim for $q=1$ to the case
  $q=0$.

The case $q > 1$ follows analogously. 
\end{proof} 

\begin{proposition}\label{prop:thh-homotopy-invariant}
  Let $D \to E$ be a $\pi_*$-isomorphism in $\Spsym(\SpO)$. If $E^m$ and $D^m$ are flat orthogonal spectra for all
  $m\geq 0$, then the induced map $\mathrm{thh}_q(D) \to
  \mathrm{thh}_q(E)$ is a $\pi_*$-isomorphism of orthogonal spectra.
\end{proposition}
\begin{proof}
  We let $m\geq 0$ be an integer and write $\Omega^{\cL}_m = \Omega^{\cL}\sh^m$. Applying $\Omega^{\cL}_m$ to $\Omega^{\cI}_{q+1}(E)$ provides a diagram
  \[\begin{split} \Omega^{\cL}_m\Omega^{\cI}_{q+1}(E)\colon &\cI^{\times (q+1)}\times \cL \to
  \cT_*,\\ &(\bld{n_0},\dots, \bld{n_q};V) \mapsto \Omega^{n_0+ \dots +
    n_q+V}((E^{n_0}\sm \dots \sm E^{n_q})(\mathbb R^m \oplus V)) .\end{split}\]
  Via the inclusions $\cN \to \cI$ and $\cN \to \cL$, it
  restricts to an $\cN^{\times(q+1)}\times \cN$-diagram of spaces. We claim
  that $D \to E$ induces a weak equivalence
\begin{equation}\label{eq:thh-homotopy-invariant}
 \hocolim_{\cN^{\times(q+1)}\times \cN} \Omega^{\cL}_m\Omega^{\cI}_{q+1}(D) \to \hocolim_{\cN^{\times(q+1)}\times \cN} \Omega^{\cL}_m\Omega^{\cI}_{q+1}(E) 
\end{equation}
when forming the homotopy colimits of these restricted diagrams. To
see this, we first consider the map of $\cN^{\times(q+1)}\times \cN$-diagrams
\begin{multline*} (n_0,\dots, n_q;k) \mapsto \Big( \Omega^{n_0+ \dots + n_q+k}((D^{n_0}\sm \dots \sm D^{n_{q-1}} \sm D^{n_q})(\mathbb R^m \oplus \mathbb{R}^{k})) \to \\ \Omega^{n_0+ \dots + n_q+k}((D^{n_0}\sm \dots \sm D^{n_{q-1}} \sm E^{n_q})(\mathbb R^m \oplus \mathbb{R}^{k})) \Big)\ .
\end{multline*}
By Lemma~\ref{lem:flat-preservation}(i),
$D^{n_0}\sm \dots \sm D^{n_{q-1}}$ is flat. Hence
Lemma~\ref{lem:smashing-with-flat-preserves-diagonal-isos} implies
that this map induces a weak equivalence when fixing
$(n_0,\dots, n_{q-1})$ and forming $\hocolim_{\cN \times \cN}$ with
respect to the remaining two $\cN$-directions. Hence it also induces an equivalence on
$\hocolim_{\cN^{\times(q+1)}\times \cN}$. In the next step, we fix $(n_0,\dots , n_{q-2}, n_q)$ and argue analogously with the map of $\cN\times \cN$-diagrams induced by $D^{n_{q-1}}\to E^{n_{q-1}}$. Continuing this process leads to a decomposition of~\eqref{eq:thh-homotopy-invariant} into a composite of $q$ weak equivalences.  

Since forming loop spaces commutes up to homotopy with mapping telescopes and homotopy colimits of spectra are formed levelwise, it follows that 
\[
 \hocolim_{\cN} \Omega^{\cL}_m \hocolim_{\cN^{\times(q+1)}} \Omega^{\cI}_{q+1}(D) \to \hocolim_{\cN} \Omega^{\cL}_m \hocolim_{\cN^{\times(q+1)}} \Omega^{\cI}_{q+1}(E) 
\]
is a weak equivalence. Thus Lemma~\ref{lem:pi-star-orthogonal-detection} implies
that $ \Omega^{\cI}_{q+1}(D) \to \Omega^{\cI}_{q+1}(E)$ induces a $\pi_*$-isomorphism after applying $\hocolim_{\cN^{\times(q+1)}}$. By the obvious generalization of ~\cite[Proposition 2.2.9]{MR1740756} from based spaces to orthogonal spectra, it follows that $ \Omega^{\cI}_{q+1}(D) \to \Omega^{\cI}_{q+1}(E)$  induces a $\pi_*$-isomorphism after applying $\hocolim_{\cI^{\times (q+1)}}$.
\end{proof}

\begin{proof}[Proof of Theorem \ref{thm:THH-model-comparison}]
  Lemma~\ref{lem:Bcy-to-thh-sh-R} shows that the first map
  in~\eqref{eq:THH-model-comparison} is a stable equivalence in every
  simplicial degree. Combining Corollary~\ref{cor:shiftandsusp} and
  Proposition~\ref{prop:thh-homotopy-invariant} implies that the
  second map in~\eqref{eq:THH-model-comparison} is a stable
  equivalence in every simplicial degree.
  
So it remains to show that each of the simplicial objects is good. For $B^{\mathrm{cy}}(A)$, the degeneracy maps are flat
  cofibrations by Lemma~\ref{lem:unit-is-flat-cof}(ii) and Lemma~\ref{lem:flat-preservation}, and hence in particular the degeneracy maps are levelwise  
  $h$-cofibrations. For the two remaining cases we apply Proposition~\ref{goodness}. 
  The orthogonal spectrum $(\sh A)^n = \sh^n A$ is flat by 
  Theorem~\ref{thm:shifts-of-flat}, and the suspension spectrum $(\Sigma^{\infty}A)^n = \Sigma^\infty A_n = A_n \sm \mathcal L_S(0,-)$ is flat because $A_n$ is cofibrant as a based space.
\end{proof}

\section{Orthogonal \texorpdfstring{$G$}{G}-spectra and geometric fixed points}\label{sec:orthogonal-G} In this section we review some results about equivariant orthogonal spectra that we will need for the definition of the cyclotomic structures in the next section. 

Let $G$ be a compact Lie group. Let $(\SpO)^G$ denote the category of orthogonal $G$-spectra, defined as orthogonal spectra with continuous left $G$-actions and equivariant maps between them.

Unless otherwise noted, orthogonal $G$-spectra will always be considered as objects of this category, i.e., they are always indexed on the trivial universe. This is not a restriction because for each complete $G$-universe $\mathcal U$, the change of universe functor $\mathcal I_{\R^\infty}^{\mathcal U}$ gives an equivalence between $(\SpO)^G$ and the category of orthogonal $G$-spectra indexed by finite-dimensional representations in $\mathcal U$~\cite[V.Theorem~1.5]{MM02}. 

For each closed normal subgroup $N \leq G$ we write
\[ \Phi^N \colon (\SpO{})^{G} \to (\SpO)^{G/N} \]
for the \textit{geometric fixed points} functor. If $X$ is an orthogonal $G$-spectrum, $\Phi^N(X)$ is defined to be the coequalizer of a diagram
\begin{equation}\label{geometric_fixed_points_definition}
\bigvee_{V,W} \mathcal L_S(W^N,-) \sm \mathcal L_S(V,W)^N \sm X(V)^N 
\rightrightarrows \bigvee_{V} \mathcal L_S(V^N,-) \sm X(V)^N.
\end{equation}
Here $V,W$ run through all finite-dimensional representations in a
complete $G$-universe $\mathcal U$, the $G/N$-space $\mathcal
L_S(V,W)^N$ is the space of $N$-fixed points that arises from
restricting the conjugate $G$-action on $\mathcal L_S(V,W)$ to $N$,
the orthogonal $G/N$-spectrum $\mathcal L_S(W^N,-)$ is the free one on the
$G/N$-space $W^N$, and the $G/N$-space $X(V)^N$ is the space of $N$-fixed
points of the $G$-space $(\mathcal I_{\R^\infty}^{\mathcal U} X)(V) =\mathcal L(\mathbb R^d,V)_+ \sm_{O(d)}
X(\mathbb R^d)$ with diagonal $G$-action and $d = \mathrm{dim} V$. The two parallel arrows are induced by composition in $\mathcal L_S$ and by the structure maps of $X$, respectively.

\begin{remark} 
In~\cite[B.10]{HHR} this functor appears under the name \textit{monoidal geometric fixed points}. The equivalence of the above definition and the one in~\cite[V.Definition~4.1]{MM02} follows from the description of the enriched left Kan extension as a coequalizer (see also ~\cite[Remark 2.11]{ABG15}).
\end{remark}

In all of our examples, $G$ will be the circle group $S^1$ or one of its finite subgroups $C_r \cong \mathbb{Z}/r\mathbb{Z}$, $r \geq 1$. We therefore fix once and for all a complete $S^1$-universe $\mathcal U$, which is automatically a complete $C_r$-universe for all $r \geq 1$, and consistently use this universe for the definition of $\Phi^N$. Under this convention, we get the convenient property that the geometric fixed points functor strictly commutes with forgetting group actions. Specifically, if $N \leq G \leq S^1$, then in the following square each horizontal functor is defined by some coequalizer system, and the square
\[ \xymatrix{
(\SpO{})^{G} \ar[r]^-{\Phi^N} \ar[d]_-*\txt{forget} & (\SpO)^{G/N} \ar[d]^-*\txt{forget} \\
(\SpO{})^{N} \ar[r]^-{\Phi^N} & \SpO{}
} \]
commutes.
This is because $N \leq G$ satisfies Condition 1 from \cite[Proposition 3.1.46]{BDS16}, \cite[Proposition 3.3.57]{Stolz_equivariant}.
We will sometimes refer to the top horizontal functor as the ``relative'' geometric fixed points $\Phi^{N}_{\mathrm{rel}}$ and to the composite through the lower left hand corner as the ``absolute'' geometric fixed points $\Phi^{N}_{\mathrm{abs}}$. 

In fact, because $G$ is abelian, it commutes with the action of $N$ and so acts on $\Phi^N X$. The action of $N \leq G$ is trivial by \cite[Proposition 3.25]{Mal17}, so $\Phi^{N}_{\mathrm{abs}}X$ inherits a $G/N$-action. As we explain in the proof of Lemma \ref{lem:compatible-Cr-realizations} below, $\Phi^{N}_{\mathrm{rel}}X$ and $\Phi^{N}_{\mathrm{abs}}X$ are isomorphic as $G/N$-spectra. We are therefore free to think of geometric fixed points in the ``absolute'' or ``relative'' sense without risk of confusion. This fact that the two notions of $\Phi^N$ coincide is also used implicitly in \cite{ABG15} and \cite{Mal17}. 

\begin{remark}
The property that $\Phi^N$ strictly commutes with forgetful functors holds more generally for certain subgroups of the $n$-torus \cite[Lemma 3.1.42 and Lemma 4.4.20]{BDS16}. For general $N \leq G$, this commutation is only known to hold on the subcategory of cofibrant spectra, because with consistent choices of universe there is a canonical map $\Phi^{N}_{\mathrm{rel}} X \to \Phi^{N}_{\mathrm{abs}} X$ of non-equivariant spectra, and this map is an isomorphism when $X$ is free and therefore also when $X$ is cofibrant.
\end{remark}

We summarize some of the properties of $\Phi^N$ needed later. 
\begin{proposition}\label{prop:properties-geom-fixed}
Let $X$ and $Y$ be orthogonal $G$-spectra, $N \leq G \leq S^1$, and $L$ be a based $G$-space.  
\begin{enumerate}[(i)]
\item The functor $\Phi^N$ preserves cobase changes along levelwise closed inclusions and sequential colimits of levelwise closed inclusions. 
\item There is a canonical natural map $\Phi^N(X) \sm \Phi^N(Y) \to \Phi^N(X \sm Y)$ that turns $\Phi^N$ into a lax symmetric monoidal functor.  
\item\label{it:geom-fixed-points-vs-suspension} There is a canonical isomorphism $\Phi^N (X \sma L) \cong (\Phi^{N} X) \sma L^N$.
\end{enumerate}
\end{proposition}
\begin{proof}
Part (i) is discussed in \cite[\S V.4]{MM02} and \cite[Proposition B.197]{HHR}, part (ii) is \cite[Proposition V.4.7]{MM02} or \cite[(B.198)]{HHR}, and part (iii) is  \cite[(B.196)]{HHR}. The uniqueness of the maps in (ii) and (iii) is from \cite[Theorem 3.20 and Remark 3.21]{Mal17} and the above observation that $\Phi^N$ may be interpreted in the ``absolute'' sense.\end{proof}

The following result about the geometric fixed points of smash powers is a key
property of the geometric fixed points functor. 
\begin{proposition}\label{prop:diagonal-iso}
Let $X$ be an orthogonal spectrum and let $r \geq 1$. Then the diagonal map $X \to \Phi^{C_r}(X^{\sm r})$ is an isomorphism 
if $X$ is flat. 
\end{proposition}
The proposition generalizes to arbitrary finite groups $G$ (see e.g.~\cite[Proposition B.209]{HHR}), but we shall not make use of this extra generality. 
\begin{proof}[Proof of Proposition~\ref{prop:diagonal-iso}]
The diagonal map is for example constructed in ~\cite[Proposition B.209]{HHR}, and the statement that
it is an isomorphism for flat $X$ can be found in \cite[Theorem 3.2.16]{BDS16} or \cite[Theorem 3.4.26]{Stolz_equivariant}.
\end{proof}

We recall that orthogonal $G$-spectra are endowed with a stable model structure~\cite[III.4.2]{MM02} in which the weak equivalences are measured by the homotopy groups
\[ \pi_k^H(X) = \begin{cases}\quad
\underset{V \subset \mathcal U}\colim\, \pi_k([\Omega^V X(V)]^H)  &\text{if } k \geq 0 \\
\underset{V \subset \mathcal U,\ \R^{|k|} \subset V}\colim\, \pi_0([\Omega^{V-\R^{|k|}} X(V)]^H)  &\text{if } k < 0,\end{cases} \]
where the colimits are taken along inclusions of finite-dimensional subspaces of $\mathcal U$, and again $\cI_{\R^\infty}^{\mathcal U}$ is used implicitly to describe the levels $X(V)$.

The geometric fixed points functors $\Phi^H$ are not left adjoints, but they preserve weak equivalences between the cofibrant spectra in this model structure. We therefore refer to $\Phi^H (X^c)$, where $X^c$ is any cofibrant replacement of $X$, as the left-derived geometric fixed points $\mathbb{L}\Phi^H(X)$. A map of $G$-spectra $X \to Y$ is a weak equivalence precisely when it induces an equivalence on the derived geometric fixed points $\mathbb{L}\Phi^H$ for all closed subgroups $H \leq G$ (e.g. \cite[Theorem XVI.6.4]{alaska}). 

For a $G$-representation $V$, we get a natural interchange morphism  
\begin{equation}\label{eq:Phi-G-vs-loops} \Phi^{G} (\Omega^V X) \to \Omega^{V^{G}} \Phi^{G} (\Sigma^V \Omega^V X) \to \Omega^{V^{G}} \Phi^{G} X \end{equation}
where the first map is the adjoint of the isomorphism in Proposition~\ref{prop:properties-geom-fixed}\eqref{it:geom-fixed-points-vs-suspension} for $L = S^V$ and the second map is the counit.
We derive the interchange by assuming that $X$ is cofibrant and by inserting a cofibrant replacement of $\Omega^V X$ as shown.
\begin{equation}\label{eq:derived-Phi-G-vs-loops} \Phi^{G} ((\Omega^V X)^c) \to \Omega^{V^{G}} \Phi^{G} (\Sigma^V (\Omega^V X)^c) \to \Omega^{V^{G}} \Phi^{G} (\Sigma^V \Omega^V X) \to \Omega^{V^{G}} \Phi^{G} X. \end{equation}
\begin{lemma}\label{lem:PhiG-vs-OmegaV} Let $X$ be an orthogonal $G$-spectrum such that a cofibrant replacement $X^c \to X$ in $(\SpO{})^G$ induces a stable equivalence $\Phi^G(X^c) \to \Phi^G(X)$. Then the composite \eqref{eq:derived-Phi-G-vs-loops} is a stable equivalence. 
\end{lemma}
\begin{proof} By the assumption on $X$ it is sufficient to show that \eqref{eq:derived-Phi-G-vs-loops} is a stable equivalent for cofibrant $X$. Since $\Omega^{V^G}$ preserves stable equivalences between all objects, the first map in \eqref{eq:derived-Phi-G-vs-loops} is a stable equivalence since it is the adjoint along a Quillen equivalence of 
an isomorphism with cofibrant source. The remaining two maps are $\Omega^{V^G} \Phi^G$ of the composite $\Sigma^V (\Omega^V X)^c \overset\sim\to \Sigma^V \Omega^V X \overset\sim\to X$. As the source and target are cofibrant $G$-spectra, we still get an equivalence after applying $\Omega^{V^G} \Phi^G$. 
\end{proof}

Below we use the theory of $G$-diagrams in the sense of \cite{DM16}. Let $G$ be a finite group and $\mathcal C$ a category with a (strict) $G$-action. We recall that a $G$-diagram of orthogonal spectra is a functor $Z \colon \mathcal C \to \SpO{}$ equipped with maps of orthogonal spectra $\phi_g\colon Z_c\to Z_{gc}$ for every $g\in G$ and every object $c$ of $\mathcal C$. These maps are natural with respect to the morphisms of $\mathcal C$ and they satisfy $\phi_1=\id$ and $\phi_h\phi_g=\phi_{hg}$ for every $g,h\in G$. We observe that in particular the spectrum $Z_c$ has an action of the stabilizer group $G_c=\{g\in G\ |\ gc=c\}$, and therefore it is an object of $(\SpO)^{G_c}$. In particular if $c$ is a fixed point, $Z_c$ is a $G$-spectrum. The homotopy colimit of $Z \colon \mathcal C \to \SpO{}$, defined as the Bousfield-Kan formula
\[\hocolim_{\mathcal C} Z=
|\bigvee\limits_{\underline{c}\in N\mathcal C} Z_{c_q}|
\]
where $\underline{c} = (c_q \to \dots \to c_0)$ denotes a $q$-simplex in the nerve of $\mathcal C$,
inherits a $G$-action from the maps $\phi_g$ (\cite[Definition 1.16]{DM16}). The action of $g\in G$ is defined as the geometric realization of the simplicial map that takes the $\underline{c}$-summand to the $g\underline{c}=(gc_q \to \dots \to gc_0)$-summand via the map $\phi_{g}$.

\begin{lemma}\label{lem:hocolim-PhiG-interchange}
Let $G$ be a finite group, $\mathcal C$ a category with $G$-action, $Z \colon \mathcal C \to \SpO{}$ a $G$-diagram, and $\Delta \colon \mathcal C^G \to \mathcal C$ the fixed points inclusion. Then there is a natural isomorphism 
\[\Phi^{G} \hocolim_{\mathcal C} Z  \xrightarrow{\iso} \hocolim_{\mathcal C^G} \Phi^{G} (Z\circ \Delta)\, .\]
\end{lemma}
\begin{proof}
The isomorphism arises as the composition of the canonical isomorphisms
\begin{eqnarray*}
\Phi^{G} \hocolim_{\mathcal C} Z&=&
\Phi^{G}|\bigvee\limits_{\underline{c}\in N\mathcal C} Z_{c_q}|\cong
|\Phi^{G}\bigvee\limits_{\underline{c}\in N\mathcal C} Z_{c_q}|\cong
|\bigvee\limits_{\underline{c}\in (N\mathcal C)^G} \Phi^{G}Z_{c_q}|
\\&=&
|\bigvee\limits_{\underline{c}\in N(\mathcal C^G)} \Phi^{G}Z_{c_q}|
=\hocolim_{\mathcal C^G} \Phi^{G} (Z\circ \Delta).
\end{eqnarray*}
For the first isomorphism we use that $\Phi^G$ commutes with geometric realizations since it preserves cobase change along closed inclusions and sequential colimits of closed inclusions. The fact that $\Phi^{G}$ commutes with indexed coproducts can be immediately verified from the coend formula defining $\Phi^{G}$.
\end{proof}

To set up a final compatibility result involving $\Phi^G$, we let
$\Lambda_r$ be the variant of the cyclic category discussed in
~\cite[Definition 1.5]{BHM93} and recall that the realization of a
$\Lambda_r$-space has a canonical $\mathbb R/r\mathbb Z$-action that
extends a $C_r$-action defined on the underlying simplicial space.

Let $Y_{\bullet} \colon \Lambda_r^{\op} \to \SpO$ be a
$\Lambda_r$-object in orthogonal spectra. Since the $C_r$-action also
commutes with the $C_{rn}$-action in simplicial degree $n-1$, we may
view $Y_{\bullet}$ as a $\Lambda_r$-object in orthogonal $C_r$-spectra. Taking
(absolute) $C_r$-geometric fixed points in each level gives rise to a
$\Lambda_r$-orthogonal spectrum $\Phi^{C_r}(Y_{\bullet})$. Its realization 
$|\Phi^{C_r}(Y_{\bullet})|$ has an $S^{1}$-action with trivial underlying 
$C_r$-action, and thus an $S^1/C_r$-action.  On the other hand, we may apply 
the (relative) $\Phi^{C_r}$ to the orthogonal $S^1$-spectrum~$|Y_{\bullet}|$. 
\begin{lemma}\label{lem:compatible-Cr-realizations}
In this situation, there is a natural $S^1/C_r$-equivariant 
isomorphism $|\Phi^{C_r}(Y_{\bullet})| \cong \Phi^{C_r}|Y_{\bullet}|$. 
\end{lemma}
\begin{proof}
We construct a homeomorphism which is $S^1$-equivariant with
respect to the $S^1$-action that restricts to a trivial $C_r$-action. 
As a first step, commuting the realization with coequalizer, smash
products and $C_r$-fixed points provides a homeomorphism 
$|\Phi^{C_r}(Y_{\bullet})| \cong \Phi^{C_r}|(Y_{\bullet})|$ where 
$\Phi^{C_r}$ is understood in the absolute sense. This homeomorphism
can be checked to be $S^1$-equivariant. In the second step, we 
note that the $S^1$-actions on the relative and the absolute
versions of  $\Phi^{C_r}|(Y_{\bullet})|$ coincide: The action of $g \in S^1$ sends a point represented by \[ ((\alpha,w),\beta, x) \in  \mathcal L_S(V^{C_r},\mathbb R^k) \sm \mathcal L(\mathbb R^{\mathrm{dim}\,V},V)_+ \sm |(Y_{\bullet})|(\mathbb R^{\mathrm{dim}\,V}) \] to  $((\alpha,w),\beta, gx)$ in the first case and to $((\alpha\circ g^{-1},w),g\circ \beta, gx)$ in the second case, and these points get identified in 
the coequalizer. 
\end{proof}

\section{Cyclotomic structures}\label{sec:cyclotomic-str}
In this section we show that the chain of stable equivalences in Theorem~\ref{thm:THH-model-comparison} is in fact a chain of equivalences of cyclotomic spectra, in the sense to be reviewed below. 
\subsection{Construction of cyclotomic structures}
We recall the following notion of cyclotomic spectra from~\cite{BM16}.
As usual $\rho_r\colon S^{1}\to S^{1}/C_r$ is the group isomorphism induced by the $r$-fold root, and $\rho_{r}^\ast$ denotes restriction along $\rho_r$.
\begin{definition}\phantomsection\label{def:cyclotomic-spectra}\begin{enumerate}[(i)]
\item A \textit{pre-cyclotomic spectrum} is an orthogonal $S^1$-spectrum $T$ with choices of maps of $S^1$-spectra
\[ c_r\colon \rho_r^*\Phi^{C_r}T \to T \] for every integer $r \geq 1$ such that following square commutes for any $r$ and $s$:
\begin{equation}\label{eq:cyclotomic-compatibility} \xymatrix{
\rho_{rs}^* \Phi^{C_{rs}} T \ar[rr]^-{c_{rs}} \ar[d]^-{it} && T \\
\rho_r^* \Phi^{C_r} \rho_s^* \Phi^{C_s} T \ar[rr]^-{\rho_r^*\Phi^{C_r} c_s} && \rho_r^* \Phi^{C_r} T \ar[u]_-{c_r} \rlap{\ .}} \end{equation}
Here $it$ is the map $c_{r,s}$ from \cite[Proposition 2.6]{BM16}; its definition is also forced from \cite[Theorem 3.23]{Mal17}.
\item A \textit{cyclotomic spectrum} is a pre-cyclotomic spectrum $T$ such that the induced map out of the derived geometric fixed points
\[ \rho_r^*\mathbb{L}\Phi^{C_r}T \to \rho_r^*\Phi^{C_r}T \stackrel{c_r}{\longrightarrow} T \]
is a stable equivalence of underlying orthogonal spectra for all $r\geq 1$. 
\end{enumerate}
\end{definition}

\begin{remark}
The definition of a cyclotomic spectrum \cite[Definition 4.10]{BM16} requires the map $\rho_r^*\mathbb{L}\Phi^{C_r}T \to \rho_r^*\Phi^{C_r}T \to T$ to be an $\mathcal{F}$-equivalence, that is, a map that induces an equivalence on derived $C_s$-geometric fixed-points for every integer $s\geq 1$. It is however sufficient to require these maps to be underlying equivalences. Indeed by condition (i) the diagram
\[
\xymatrix{
\rho_s^*\mathbb{L}\Phi^{C_s}\rho_r^*\mathbb{L}\Phi^{C_r}T
\ar[rr]\ar[d]_{\simeq}
&& \rho_s^*\mathbb{L}\Phi^{C_s}\rho_r^*\Phi^{C_r}T\ar[rr]^-{\rho_s^*\mathbb{L}\Phi^{C_s}c_r}\ar[d]
&&
\rho_s^*\mathbb{L}\Phi^{C_s}T\ar[d]
\\
\rho_s^*\Phi^{C_s}\rho_r^*\mathbb{L}\Phi^{C_r}T
\ar[rr]
&& \rho_s^*\Phi^{C_s}\rho_r^*\Phi^{C_r}T\ar[rr]^-{\rho_s^*\Phi^{C_s}c_r}
&&
\rho_s^*\Phi^{C_s}T\ar[d]^{c_s}
\\
\rho_{sr}^*\mathbb{L}\Phi^{C_{sr}}T\ar[u]^-{\simeq}_{it}
\ar[rr]
&&\rho_{sr}^*\Phi^{C_{sr}}T\ar[rr]^-{c_{sr}}\ar[u]_{it}
&&
T
}
\]
commutes, where the upper left vertical map is an equivalence since $\rho_r^*\mathbb{L}\Phi^{C_r}T$ is a cofibrant $S^1$-spectrum, and the lower left vertical map is an equivalence by \cite[Proposition 2.6]{BM16}. By condition (ii) the lower horizontal and the right vertical composites are equivalences, and therefore so is the composite of the top row.

We also recall from~\cite[Proposition 5.5]{BM16} that if $T$ and $T'$ are cyclotomic spectra, any stable equivalence $T \to T'$ of underlying non-equivariant spectra commuting with the maps $c_r$ is automatically an $\mathcal F$-equivalence.
\end{remark}

%, meaning that it induces equivalences on the (derived) geometric $C_r$-fixed points and on the genuine $C_r$-fixed points for all $r$.

Let $E$ be a symmetric spectrum object in $\SpO{}$ such that $E^n$ is a flat orthogonal spectrum for all $n \geq 0$. As usual, we write $\mathrm{sd}_r$ for the $r$-fold edgewise subdivision. Now consider the composite 
\begin{equation} 
\label{eq:restriction-on-thhEq}
\begin{split}
\gamma_r\colon & \Phi^{C_r} (\mathrm{sd}_r \mathrm{thh}_{\bullet}(E))_q 
\xrightarrow{=} \Phi^{C_r} \hocolimsubscr_{\cI^{\times (q+1)r}} \Omega^{\cI}_{(q+1)r}(E)
\\
& \xrightarrow{\cong}\; 
\hocolimsubscr_{\cI^{\times (q+1)}} \Phi^{C_r} (\Omega^{\cI}_{(q+1)r}(E) \circ \Delta_r) \\  & \xrightarrow{=}\; 
  \hocolimsubscr_{(\bld{n_0},\dots, \bld{n_q}) \in \cI^{\times (q+1)}} \Phi^{C_r} \Omega^{r(n_0+ \dots + n_q)}\left((E^{n_0}\sm \dots \sm E^{n_q})^{\sm r}\right)\\ & \xrightarrow{\lambda}\; 
\hocolimsubscr_{(\bld{n_0},\dots, \bld{n_q}) \in \cI^{\times (q+1)}} \Omega^{n_0+ \dots + n_q}\left(\Phi^{C_r}(E^{n_0}\sm \dots \sm E^{n_q})^{\sm r}\right)\\ & \xrightarrow{\cong}\;  
\hocolimsubscr_{(\bld{n_0},\dots, \bld{n_q}) \in \cI^{\times (q+1)}} \Omega^{n_0+ \dots + n_q}(E^{n_0}\sm \dots \sm E^{n_q}).
\end{split}
\end{equation}
Here $\Delta_r$ stands for the diagonal $\cI^{\times (q+1)} \to \cI^{\times (q+1)r}$, which is the inclusion of the fixed-points category of $\cI^{\times (q+1)r}$ under the $C_r$-action.  The first isomorphism is the canonical interchange discussed in Lemma~\ref{lem:hocolim-PhiG-interchange}. The map $\lambda$ is an instance of the interchange map~\eqref{eq:Phi-G-vs-loops}. 
It is a straightforward but tedious diagram-chase to verify that $\lambda$ defines a map of diagrams over $\cI^{\times (q+1)}$. Finally the last map above is induced by the inverse of the diagonal isomorphism  from Proposition~\ref{prop:diagonal-iso} for the flat orthogonal spectrum $X = E^{n_0}\sm \dots \sm E^{n_q}$. Note that $\Phi^{C_r}$ is taken in the ``absolute'' sense when defining these maps.

For the statement of the following result we use the fact that $\Phi^{C_r} \sd_r(-)$ preserves cyclic objects in orthogonal spectra, see e.g. \cite[Theorem 4.6]{ABG15} or \cite[Proposition 4.1]{Mal17}.

\begin{proposition} \label{cyclic map} Let $E$ be a ring bispectrum. The map $\gamma_r$ from (\ref{eq:restriction-on-thhEq}) is a morphism of cyclic objects.
\end{proposition}

\begin{proof}

We claim that the naturality of the isomorphism in Lemma~\ref{lem:hocolim-PhiG-interchange} defines uniquely a cyclic structure on 
\[\hocolimsubscr_{\cI^{\times (q+1)}} \Phi^{C_r} (\Omega^{\cI}_{(q+1)r}(E) \circ \Delta_r)\] 
so that the first isomorphism in (\ref{eq:restriction-on-thhEq}) is cyclic. For example, in the case of the face operators, we have the following commutative diagram
\[\xymatrix@C=10pt{\hocolimsubscr_{\cI^{\times (q+1)}} \Phi^{C_r} (\Omega^{\cI}_{(q+1)r}(E) \circ \Delta_r) \ar[r]^-{\cong} \ar[d]^{\hocolimsubscr_{\cI^{\times (q+1)}} \Phi^{C_r}(d_i^{[r]} \Delta_r)} & \Phi^{C_r} \hocolimsubscr_{\cI^{\times (q+1)r}} \Omega^{\cI}_{(q+1)r}(E)  \ar[d]^{\Phi^{C_r} \hocolimsubscr_{\cI^{\times (q+1)r}} (d_i^{[r]})} \\ \hocolimsubscr_{\cI^{\times (q+1)}} \Phi^{C_r} (\Omega^{\cI}_{qr}(E) \circ d_i^{[r]} \circ \Delta_r )  \ar[r]_-\cong \ar[d]^{can} & \Phi^{C_r} \hocolimsubscr_{\cI^{\times (q+1)r}} (\Omega^{\cI}_{qr}(E) \circ d_i^{[r]}) \ar[d]^{\Phi^{C_r}(can)}   \\ \hocolimsubscr_{\cI^{\times q}} \Phi^{C_r} (\Omega^{\cI}_{qr}(E) \circ \Delta_r) \ar[r]^\cong & \Phi^{C_r} \hocolimsubscr_{\cI^{\times qr}} \Omega^{\cI}_{qr}(E)\rlap{\ .}}\]
Here $d_i^{[r]}$ stands for the $i$-th face operator of the $r$-fold subdivision. In fact, we abuse here notation and use $d_i^{[r]}$ for the natural $C_r$-transformation  of $C_r$-diagrams
\[d_i^{[r]}\colon \Omega^{\cI}_{(q+1)r}(E) \to \Omega^{\cI}_{qr}(E) \circ d_i^{[r]}  \]
and for the $C_r$-equivariant functor $d_i^{[r]}\colon \cI^{\times (q+1)r} \to \cI^{\times qr}$. The symbol $can$ denotes the change of diagrams transformation.
We also note that $d_i^{[r]} \circ \Delta_r= \Delta_r \circ d_i$. The vertical composites are the $i$-th face operators. All the other cyclic structure maps are handled similarly, including the cyclic operators.

The next step is to check that the map 
\[\hocolimsubscr_{\cI^{\times (q+1)}} \Phi^{C_r} (\Omega^{\cI}_{(q+1)r}(E) \circ \Delta_r) \to  \hocolimsubscr_{(\bld{n_0},\dots, \bld{n_q}) \in \cI^{\times (q+1)}} \Omega^{n_0+ \dots + n_q}(E^{n_0}\sm \dots \sm E^{n_q})\]
is cyclic. Compatibility with the degeneracy maps and all the face maps  except the last one is straightforward. We will now show that the latter map respects the cyclic operators. This will imply the desired result.

We need to show that the diagram
\[\xymatrix{\hocolimsubscr_{\cI^{\times (q+1)}} \Phi^{C_r} (\Omega^{\cI}_{(q+1)r}(E) \circ \Delta_r) \ar[r] \ar[d]^{\hocolimsubscr_{\cI^{\times (q+1)}} \Phi^{C_r}(t^{[r]}_q \Delta_r)} & \hocolimsubscr_{\cI^{\times (q+1)}} \Omega^{\cI}_{q+1}(E)  \ar[d]^{\hocolimsubscr_{\cI^{\times (q+1)}} (t_q)} \\ \hocolimsubscr_{\cI^{\times (q+1)}} \Phi^{C_r} (\Omega^{\cI}_{(q+1)r}(E) \circ t^{[r]}_q \circ \Delta_r )  \ar[r] \ar[d]^{can} & \hocolimsubscr_{\cI^{\times (q+1)}} (\Omega^{\cI}_{q+1}(E) \circ t_q) \ar[d]^{can}   \\ \hocolimsubscr_{\cI^{\times (q+1)}} \Phi^{C_r} (\Omega^{\cI}_{(q+1)r}(E) \circ \Delta_r) \ar[r] & \hocolimsubscr_{\cI^{\times (q+1)}} \Omega^{\cI}_{q+1}(E). }\]
\noindent commutes, where the vertical maps arise from the cyclic structure and the horizontal maps arise from \eqref{eq:restriction-on-thhEq}. Here $t^{[r]}_q$ denotes  the $\Lambda_r$-cyclic operator for the $r$-fold subdivision in degree $q$.  More precisely, we have a functor $t^{[r]}_q \colon {\cI}_{(q+1)r} \to {\cI}_{(q+1)r}$ and a natural transformation
\[t^{[r]}_q \colon \Omega^{\cI}_{(q+1)r}(E) \to \Omega^{\cI}_{(q+1)r}(E) \circ t^{[r]}_q. \]
Also note that  $t^{[r]}_q \circ \Delta_r= \Delta_r \circ t_q$. 

The commutativity of the lower square follows  from the naturality of the canonical change of diagrams map $can$. The commutativity of the upper square requires an argument. In fact the diagram commutes before passing to homotopy colimits, in other words, the diagram\[\xymatrix{ \Phi^{C_r} \Omega^{r(n_0+ \dots + n_q)}(E^{n_0}\sm \dots \sm E^{n_q})^{\sm r} \ar[r] \ar[d]^{\Phi^{C_r}(t^{[r]}_q \Delta_r)} & \Omega^{n_0+ \dots + n_q}(E^{n_0}\sm \dots \sm E^{n_q}) \ar[d]^{t_q} \\ \Phi^{C_r} \Omega^{r(n_q+ \dots + n_{q-1})} (E^{n_q}\sm \dots \sm E^{n_{q-1}})^{\sm r} \ar[r] &  \Omega^{n_q+ \dots + n_{q-1}}(E^{n_q}\sm \dots \sm E^{n_{q-1}}) }\]
is commutative. Recall that the horizontal maps are induced from the natural transformation $\Phi^G(\Omega^V X) \to \Omega^{V^G}\Phi^G X$ and the diagonal isomorphism $X \to \Phi^{C_r}(X^{\sm r})$. When analyzing the cyclic permutations on the left and right, we see that the one on left is the composition of the $r$-fold smash power of the one on the right and a cyclic permutation of the $n_q$ coordinates. The $C_r$-geometric fixed points of such a cyclic permutation of smash factors is the identity by \cite[Proposition 3.26]{Mal17} (see also \cite[Lemma 4.5]{ABG15}). Moreover $C_r$-fixed points of the cyclic permutation of the $r$-fold sum $\mathbb{R}^{rn} =\mathbb{R}^{n} \oplus \cdots \oplus \mathbb{R}^{n}$ is also the identity. This completes the proof by the naturality of $\Phi^G(\Omega^V X) \to \Omega^{V^G}\Phi^G X$ and the diagonal isomorphism. \end{proof}

As a consequence of the latter proposition we get an $S^1$-equivariant map
\begin{equation}\label{eq:cyc-str-part-one}|\gamma_r|\colon|\Phi^{C_r} (\mathrm{sd}_r \mathrm{thh}_{\bullet}(E))| \rightarrow |\mathrm{thh}_{\bullet}(E))|=\mathrm{thh}(E).\end{equation}

\begin{proposition}\label{prop:geomandsub}
There is an isomorphism of $S^1$-spectra
\begin{equation}\label{eq:cyc-str-part-two}|\Phi^{C_r} (\mathrm{sd}_r \mathrm{thh}_{\bullet}(E))| \xrightarrow{\iso} \rho_r^*\Phi^{C_r} | \mathrm{thh}_{\bullet}(E)|. \end{equation}
\end{proposition}
\begin{proof}
This is~\cite[Proposition 4.1]{Mal17}. The middle isomorphism
$|P_r \Phi^{C_r} \mathrm{sd}_r X_{\bullet}| \iso \Phi^{C_r}|\mathrm{sd}_r X_{\bullet}|$
appearing in that argument is a consequence of Lemma~\ref{lem:compatible-Cr-realizations}. 
The isomorphism in the proposition also appears in \cite[Theorem 4.6]{ABG15}.
\end{proof}
Combining the two morphisms~\eqref{eq:cyc-str-part-one} and \eqref{eq:cyc-str-part-two} we get $S^1$-equivariant maps
\[c_r\colon\rho_r^*\Phi^{C_r}\mathrm{thh}(E) \rightarrow \mathrm{thh}(E)\] for all $r \geq 1$.  

\begin{proposition} \label{compatibility}
The maps $c_r$ define a pre-cyclotomic structure on $\mathrm{thh}(E)$. 
\end{proposition}

\begin{remark} When $E = \Sigma^\infty R$ for an orthogonal ring spectrum $R$ this proposition is \cite[Theorem 4.9]{BM12}, since the maps constructed above coincide with the maps constructed in \cite[Section 4]{BM12}. To construct the pre-cyclotomic structure maps in \cite{BM12} the authors use the original approach from \cite{HM97}, which specifies the morphisms levelwise and then Kan extends to pass to geometric fixed points. Our approach works entirely in the category of orthogonal spectra without specifying the levels, and we construct the pre-cyclotomic structure maps as the realization of cyclic maps (see Remark \ref{notcyclic} below). In our comparison theorem we will need Proposition \ref{compatibility} for $E = \sh{}R$.
\end{remark}

\begin{proof}[Proof of Proposition \ref{compatibility}]
Let us write $X_\bullet:=\mathrm{thh}_\bullet(E)$.
We have to show that diagram~\eqref{eq:cyclotomic-compatibility} commutes. This is the diagram 
\[
\xymatrix @C=6em{
|X_\bullet| & |\Phi^{C_{rs}} \sd_{rs} X_\bullet| \ar[l]_-{|\gamma_{rs}|} \ar[d]^-{|it|} \ar[r]_-\cong & \rho_{rs}^{\ast}\Phi^{C_{rs}} |X_\bullet| \ar[d]^-{it} \\
|\Phi^{C_r} \sd_r X_\bullet| \ar[u]_-{|\gamma_{r}|} \ar[d]^-{}_-\cong & |\Phi^{C_r} \Phi^{C_s} \sd_{rs} X_\bullet| \ar[l]_-{|\Phi^{C_r} \sd_r \gamma_s|} \ar[d]^-{}_-\cong \ar[r]^-{}_-\cong & \rho_{r}^{\ast}\Phi^{C_r} \rho_{s}^{\ast}\Phi^{C_s} |X_\bullet| \ar@{=}[d] \\
\rho_{r}^{\ast}\Phi^{C_r} |X_\bullet| & \rho_{r}^{\ast}\Phi^{C_r} |\Phi^{C_s} \sd_s X_\bullet| \ar[l]_-{\rho_{r}^{\ast}\Phi^{C_r} |\gamma_{s}|} \ar[r]^-{}_-\cong & \rho_{r}^{\ast}\Phi^{C_r} \rho_{s}^{\ast}\Phi^{C_s} |X_\bullet| \rlap{\ ,}}
\]
where the unlabeled isomorphisms are from Proposition \ref{prop:geomandsub} (we are also using that $\Phi^{C_r} \Phi^{C_s} \sd_{rs}=\Phi^{C_r} \sd_{r}\Phi^{C_s} \sd_{s}$).
The proof from \cite[Theorem 4.6]{Mal17} takes care of all but the top-left region. But all of the maps in that region are realizations of simplicial maps, so it suffices to show that the simplicial maps themselves commute:
\[
\xymatrix @C=7em{
X_{q} & \Phi^{C_{rs}} X_{rs(q+1)-1} \ar[l]_-{\gamma_{rs}} \ar[d]^-{it} \\
\Phi^{C_r} X_{r(q+1)-1} \ar[u]^-{\gamma_r} & \Phi^{C_r} \Phi^{C_s} X_{rs(q+1)-1} \ar[l]_-{\Phi^{C_r} \sd_r \gamma_s}\rlap{\ .}
}
\]
Remembering that $X_{q}$ is the $(q+1)$-fold B\"okstedt smash product
\[ \hocolimsubscr_{\cI^{\times (q+1)}} \Omega^{\cI}_{(q+1)}(E) = \hocolimsubscr_{(\bld{n_0},\ldots,\bld{n_{q}}) \in \cI^{\times (q+1)}} \Omega^{n_0+\ldots+n_{q}} (E^{n_0} \sma \ldots \sma E^{n_{q}}) \]
for a ring bispectrum $E$, this final square may be decomposed in the following way (where we omit the decorations of $\Omega$ and $E$ to improve readability): 
\[
\xymatrix@C=14pt@R=9pt{ 
  & & \Phi^{C_{rs}} \displaystyle\hocolimsubscr_{\cI^{\times rs(q+1)}} \Omega^{\dotsm}E^{\dotsm}\!\! \ar[d]^-\cong \ar@<10ex>@/^1.5pc/[ddddd]_-{it} \ar@<-2ex>[dll]_-{\gamma_{rs}}\\
  \displaystyle\hocolimsubscr_{\cI^{\times (q+1)}}  \Omega^{\dotsm}E^{\dotsm}\! \ar[r]^-{\delta_{rs}}_-\cong \ar[d]_-{\delta_r}^-\cong  &
  \displaystyle\hocolimsubscr_{\cI^{\times (q+1)}} \Omega^{\dotsm} \Phi^{C_{rs}} E^{\dotsm}\! \ar[d]^-{it} &
  \displaystyle\hocolimsubscr_{\cI^{\times (q+1)}} \Phi^{C_{rs}}\Omega^{\dotsm}E^{\dotsm}\!\! \ar[l]_-\lambda \ar[dd]^-{it} \\
  \displaystyle\hocolimsubscr_{\cI^{\times (q+1)}} \Omega^{\dotsm} \Phi^{C_r} E^{\dotsm}\! \ar[r]^-{\delta_s}_-\cong &
  \displaystyle\hocolimsubscr_{\cI^{\times (q+1)}} \Omega^{\dotsm} \Phi^{C_r}\Phi^{C_s} E^{\dotsm} &
   \\
  \displaystyle\hocolimsubscr_{\cI^{\times (q+1)}} \Phi^{C_r} \Omega^{\dotsm}E^{\dotsm}\! \ar[u]^-\lambda \ar[r]^-{\delta_s}_-\cong &
  \displaystyle\hocolimsubscr_{\cI^{\times (q+1)}} \Phi^{C_r} \Omega^{\dotsm} \Phi^{C_s} E^{\dotsm} \ar[u]^-\lambda &
  \displaystyle\hocolimsubscr_{\cI^{\times (q+1)}}  \Phi^{C_r} \Phi^{C_s}\Omega^{\dotsm}E^{\dotsm}\!\! \ar[l]_-\lambda  \\
 \Phi^{C_r}  \displaystyle\hocolimsubscr_{\cI^{\times r(q+1)}} \Omega^{\dotsm}E^{\dotsm}\! \ar[u]^-\cong \ar[r]^-{\delta_s}_-\cong \ar@<8ex>@/^1pc/[uuu]^-{\gamma_r}&
 \Phi^{C_r}  \displaystyle\hocolimsubscr_{\cI^{\times r(q+1)}} \Omega^{\dotsm} \Phi^{C_s} E^{\dotsm} \ar[u]^-\cong &
 \Phi^{C_r}  \displaystyle\hocolimsubscr_{\cI^{\times r(q+1)}} \Phi^{C_s}\Omega^{\dotsm}E^{\dotsm}\!\! \ar[u]^-\cong \ar[l]_-\lambda \\
 && \Phi^{C_r} \Phi^{C_s} \displaystyle\hocolimsubscr_{\cI^{\times rs(q+1)}} {\Omega^{\dotsm}}E^{\dotsm}\!\! \ar[u]^-\cong \ar@<2ex>[ull]^-{\Phi^{C_r} \sd_r \gamma_s}
}
\]
where $\delta_r\colon X \to \Phi^{C_r}(X^{\sm r})$ is the diagonal isomorphism from Proposition \ref{prop:diagonal-iso}.
The upper-left hand square commutes by \cite[Proposition 3.29]{Mal17}. The remaining regions are formal diagram-chases in which \cite[Theorem 1.2]{Mal17} occasionally saves some effort.
\end{proof}

 Our next task is to show that this pre-cyclotomic structure is actually cyclotomic:
\begin{proposition}\label{prop:pre-cyc-structures-are-cyc} Let $E$ be a ring bispectrum such that $E^n$ is a flat orthogonal spectrum for all $n \geq 0$. Then the above defined pre-cyclotomic structure on $\mathrm{thh}(E)$ is cyclotomic.
\end{proposition}
\begin{proof} 
We need to prove that the composite
\[\mathbb{L}\Phi^{C_r} | \mathrm{thh}_{\bullet}(E)| \rightarrow \Phi^{C_r} | \mathrm{thh}_{\bullet}(E)| \stackrel{c_r}{\longrightarrow} |\mathrm{thh}_{\bullet}(E)| \]
is a stable equivalence of orthogonal spectra, where $\mathbb{L}\Phi^{C_r}$ stands for the derived geometric fixed points. 

Let $P_{\bullet} \rightarrow \mathrm{sd}_r \mathrm{thh}_{\bullet}(E)$ denote a projective cofibrant replacement in the Reedy model structure of simplicial $C_r$-spectra, built from the level model structure of $C_r$-spectra. Then the geometric realization $|P_{\bullet}|$ is cofibrant as a $C_r$-spectrum. To argue that $|P_\bullet| \to |\mathrm{sd}_r \mathrm{thh}_{\bullet}(E)|$ is a level equivalence of $C_r$-spectra, it suffices to show for each $k | r$ that both $P_\bullet^{C_{r/k}}$ and $\mathrm{sd}_r \mathrm{thh}_{\bullet}(E)^{C_{r/k}}$ are good, meaning that each degeneracy is a levelwise $h$-cofibration of spectra. For $\mathrm{sd}_r \mathrm{thh}_{\bullet}(E)^{C_{r/k}}$, we use essentially the same argument as \cite[Lemma 6.17(ii)]{LRRV17}, but with the extra ingredient that an $r/k$-fold smash power of a flat cofibration of orthogonal spectra is a $C_{r/k}$-flat cofibration \cite[3.4.23]{Stolz_equivariant}, and therefore has the $C_{r/k}$-equivariant homotopy extension property on each spectrum level. The goodness of $P_\bullet^{C_{r/k}}$ also follows from the fact that $C_r$-cofibrations of spectra have the $C_r$-equivariant homotopy extension property on each spectrum level.

Therefore $|P_{\bullet}| \to | \mathrm{thh}_{\bullet}(E)|$ is a level equivalence of $C_r$-spectra, so $\Phi^{C_r}|P_{\bullet}|$ models the derived geometric fixed points $\mathbb{L}\Phi^{C_r} | \mathrm{thh}_{\bullet}(E)|$. There is a commutative diagram
\[\xymatrix{|\Phi^{C_r}P_{\bullet}| \ar[r] \ar[d]^\cong & |\Phi^{C_r} (\mathrm{sd}_r \mathrm{thh}_{\bullet}(E))| \ar[d]^\cong \\ \Phi^{C_r}|P_{\bullet}| \ar[r] & \Phi^{C_r} |\mathrm{sd}_r \mathrm{thh}_{\bullet}(E)| }  \]
and it therefore suffices to show that the composite   
\[ |\Phi^{C_r}P_{\bullet}| \rightarrow |\Phi^{C_r} (\mathrm{sd}_r \mathrm{thh}_{\bullet}(E))|  \rightarrow  |\mathrm{thh}_{\bullet}(E)|=\mathrm{thh}(E) \]
is a weak equivalence. 

The simplicial object $\mathrm{thh}_{\bullet}(E)$ is good by Proposition~\ref{goodness}. To see that $\Phi^{C_r}P_{\bullet}$ is a good simplicial object we observe that the Reedy cofibrant object $P_{\bullet}$ on degeneracy maps provides $C_r$-cofibrations, which $\Phi^{C_r}$ takes to cofibrations of orthogonal spectra, which in particular are levelwise $h$-cofibrations. Since both simplicial objects are good, it is sufficient to show that for any $q \geq 0$, the map
\[\Phi^{C_r}P_q \rightarrow \Phi^{C_r} (\mathrm{sd}_r \mathrm{thh}_{q}(E)) \rightarrow \mathrm{thh}_{q}(E) \]
is a weak equivalence. Since a Reedy cofibration gives a cofibration in each simplicial level, $P_q \rightarrow  \mathrm{sd}_r \mathrm{thh}_{q}(E)$ is a cofibrant replacement in $C_r$-spectra. Since $\Phi^{C_r}$ preserves equivalences between cofibrant $C_r$-spectra, we may replace $P_q$ by any cofibrant approximation of the individual orthogonal $C_r$-spectrum $ \mathrm{sd}_r \mathrm{thh}_{q}(E)$.
The cofibrant replacement we use is the homotopy colimit of a cofibrant replacement of the diagram $\Omega^{\cI}_{(q+1)r}(E)$ in the $C_r$-projective model structure on equivariant diagrams of \cite[Theorem 2.6]{DM16}. Let us denote this cofibrant replacement by $(\Omega^{\cI}_{(q+1)r}(E))^c \xrightarrow{\sim} \Omega^{\cI}_{(q+1)r}(E)$. Then we have a commuting diagram
\[ \xymatrix@R=18pt{\displaystyle
\Phi^{C_r}\hocolimsubscr_{\cI^{\times (q+1)r}} (\Omega^{\cI}_{(q+1)r}(E))^c \ar[r] 
&
\displaystyle \Phi^{C_r} \hocolimsubscr_{\cI^{\times (q+1)r}} \Omega^{\cI}_{(q+1)r}(E) 
\\
\displaystyle\hocolimsubscr_{\cI^{\times (q+1)}} \Phi^{C_r}((\Omega^{\cI}_{(q+1)r}(E))^c\circ \Delta_r) \ar@{<-}[u]^\cong \ar[r] \ar@{}[rd]^(.2){}="a"^(.8){}="b" \ar "a";"b"^\sim 
&
\displaystyle\hocolimsubscr_{\cI^{\times (q+1)}} \Phi^{C_r} (\Omega^{\cI}_{(q+1)r}(E)\circ \Delta_r)  \ar@{<-}[u]^\cong  \ar@{}[d]^(.2){}="a"^(.8){}="b"^(.5){\lambda} \ar "a";"b" 
\\
&
\hspace{-2cm} \displaystyle\hocolimsubscr_{(\bld{n_0},\dots, \bld{n_q}) \in   \cI^{\times (q+1)}} \Omega^{n_0+ \dots + n_q}\left(\Phi^{C_r}(E^{n_0}\sm \dots \sm E^{n_q})^{\sm r}\right)
\\
&
\hspace{-2cm} \displaystyle\hocolimsubscr_{(\bld{n_0},\dots, \bld{n_q}) \in \cI^{\times (q+1)}} \Omega^{n_0+ \dots + n_q}(E^{n_0}\sm \dots \sm E^{n_q})\rlap{\ .}\ar[u]^-{\delta_r}_-\cong
 }\]
We remark that by \cite[Proposition 2.10]{DM16} the values of the diagram $\Omega^{\cI}_{(q+1)r}(E)^c$ at the objects of $\cI^{\times (q+1)}\cong (\cI^{\times (q+1)r})^{C_r}$ are cofibrant orthogonal $C_r$-spectra.
To see that the diagonal map is a stable equivalence, we apply Lemma~\ref{lem:PhiG-vs-OmegaV}. It reduces the claim to showing that
$\Phi^{C_r}$ captures the homotopy type of its left derived functor when
evaluated on $(E^{n_0}\sm \dots \sm E^{n_q})^{\sm r}$. This follows from
our flatness assumptions and \cite[Theorem 3.2.14 and Proposition 4.5.14]{BDS16}, which imply that this $C_r$-spectrum is built out of induced regular $C_r$-cells in the sense of \cite[Theorem 3.2.14]{BDS16} and \cite[Proposition 3.4.25]{Stolz_equivariant}.  The composite which goes through the upper right corner is exactly the map we are interested in.\end{proof}

We are now ready to prove the first main theorem from the introduction.
\begin{proof}[Proof of Theorem~\ref{thm:comparison-in-introduction}]
Theorem~\ref{thm:THH-model-comparison} provides a chain of stable equivalences
\[ B^{\mathrm{cy}}(A) \to  \mathrm{thh}(\sh{}A) \ot \THH(A),\] 
and it remains to show that these are maps of cyclotomic spectra. The cyclotomic structure on $\mathrm{thh}(\sh{}A)$ and
$\THH(A) = \mathrm{thh}(\Sigma^{\infty}A)$ was established in Proposition~\ref{prop:pre-cyc-structures-are-cyc} (whose hypotheses were already checked in the proof of Theorem~\ref{thm:THH-model-comparison}); the naturality of this construction with respect to $E$ demonstrates that the right-hand map of our zig-zag commutes with the cyclotomic structures. Finally we observe that when the cyclotomic structure map defined at the beginning of this section is restricted to the $(\bld{0},\dots, \bld{0}) \in \cI^{\times (q+1)}$ term, it becomes just the norm diagonal, which gives the cyclotomic structure map for $B^{\mathrm{cy}}(A)$~\cite{ABG15}; this takes care of the left-hand map.
\end{proof}

\begin{remark} The final stage of the proof of Proposition \ref{prop:pre-cyc-structures-are-cyc}, and the proof of Theorem \ref{thm:THH-model-comparison}, give equivalences on every simplicial level, not just on the realization. They also do not make use of any face or degeneracy maps, and hence do not use the ring structure on $A$. Thus we have also obtained a new proof of the main result of \cite{MR3556283}: for each flat orthogonal spectrum $X$ and integer $k \geq 1$, there is a natural zig-zag of genuine equivalences of orthogonal $C_k$-spectra $X^{\sma k} \simeq \mathrm{thh}_{k-1}(\Sigma^\infty X)$. In slightly more detail, the maps of our zig-zag
\[ X^{\sma k} \to \mathrm{thh}_{k-1}(\sh{}X) \ot \mathrm{thh}_{k-1}(\Sigma^\infty X) \]
are nonequivariant equivalences by the proof of Theorem \ref{thm:THH-model-comparison}, and equivalences on derived geometric fixed points by a reduction to the nonequivariant case. The reduction is by the following commuting diagram, in which the vertical equivalences are from the last stage of the proof of Proposition \ref{prop:pre-cyc-structures-are-cyc}.
\[ \xymatrix{
\mathbb{L}\Phi^{C_r} X^{\sma rs} \ar[r] \ar[d]^-\sim & \mathbb{L}\Phi^{C_r} \mathrm{thh}_{rs-1}(\sh{}X) \ar[d]^-\sim & \mathbb{L}\Phi^{C_r} \mathrm{thh}_{rs-1}(\Sigma^\infty X) \ar[l] \ar[d]^-\sim \\
X^{\sma s} \ar[r]^-\sim & \mathrm{thh}_{s-1}(\sh{}X) & \mathrm{thh}_{s-1}(\Sigma^\infty X) \ar[l]_-\sim
\rlap{\ .}} \]
\end{remark}

\section{Comparison of TC models}\label{sec:TC-comparison}
Let $A$ be an orthogonal ring spectrum. In this section we show that under suitable assumptions on $A$, the topological cyclic homology spectrum resulting from the cyclotomic structure on $B^{\mathrm{cy}}(A)$ is equivalent to the original topological cyclic homology spectrum built from $\THH(A)$. We start by restricting our attention to the underlying $p$-cyclotomic spectra, i.e., to the fixed points of $p$-power cyclic subgroups.  

Let $p$ be a prime and let $X$ be a pre-cyclotomic spectrum in the sense of Definition~\ref{def:cyclotomic-spectra}. The cyclotomic structure maps yield $S^1$-equivariant maps 
\begin{equation}\label{eq:R-maps-via-geometric-fixed-points} R_{p^m,p^n} \colon \rho_{p^{m+n}}^*X^{C_{p^{m+n}}} \to \rho^*_{p^n}X^{C_{p^n}}
\end{equation}
defined as the composite
\[ \rho_{p^{m+n}}^*X^{C_{p^{m+n}}} \to \rho^*_{p^n}(\rho_{p^m}^*\Phi^{C_{p^m}} X)^{C_{p^n}} \to \rho^*_{p^n}X^{C_{p^{n}}}, \]
where the first map is induced by the canonical map from the fixed points to the geometric fixed points \cite[\S V.4]{MM02} and the second map is the $C_{p^n}$-fixed points of the cyclotomic structure map. We also have non-equivariant maps
\begin{equation}\label{eq:F-maps} F_{p^{m},p^{n}} \colon X^{C_{p^{m+n}}} \to X^{C_{p^n}}\end{equation}
arising from the inclusion of fixed points of a larger group to a smaller group.

Let $\mathbb I$ be the category with objects $[r]$ for every natural number $r \geq 1$ and with  morphisms generated by $R_{r,s}$ and $F_{r,s}\colon [rs] \to [s]$ subject to the relations
\[ R_{1,s} = F_{1,s} = \id_{[s]}, \quad R_{s,t}R_{r,st} = R_{rs,t}, \]
\[ R_{s,t}F_{r,st} = F_{s,t}R_{r,st}, \quad  F_{s,t}F_{r,st} = F_{rs,t} \]
(see e.g. \cite{BM16} and \cite{madsen_survey}). Let $\mathbb I_p$ be the full subcategory of $\mathbb I$ on the objects $[p^m]$ with $m\geq 0$. Then the maps~\eqref{eq:R-maps-via-geometric-fixed-points} and~\eqref{eq:F-maps} define an $\mathbb I_p$-diagram sending $[p^n]$ to $X^{C_{p^n}}$, and the  $p$-typical topological cyclic homology of $X$ is defined as \[ \TC(X;p) = \holim_{F,R} X^{C_{p^n}}.\]
%where the homotopy limit is taken over the category with an object $[n]$ for every natural number $n \geq 0$, and whose morphisms are generated by $R_{m,n}$ and $F_{m,n}\colon [mn] \to [n]$ subject to the relations
%\[ R_{1,n} = F_{1,n} = \id_{[n]}, \quad R_{n,k}R_{m,nk} = R_{mn,k}, \]
%\[ R_{n,k}F_{m,nk} = F_{m,k}R_{n,mk}, \quad  F_{n,k}F_{m,nk} = F_{mn,k} \]
%(see e.g. \cite{BM16} and \cite{madsen_survey}). 
This is only a homotopy invariant notion if $X$ is a $C_{p^n}$-equivariant $\Omega$-spectrum for all $n\geq 0$. To enforce this condition, we write $X \to X^{\mathrm{fib}}$ for the fibrant replacement in the model${}^*$ structure on pre-cyclotomic spectra constructed in \cite[\S 5]{BM16} and note that $X^{\mathrm{fib}}$ has the desired $\Omega$-spectrum property. In particular, we can define the topological cyclic homology of a flat orthogonal ring spectrum $A$ by the formula $\TC((B^{\mathrm{cy}}(A))^{\mathrm{fib}};p)$; see \cite[\S 3.2]{ABG15}.

We will recall in Section~\ref{subsec:classic-cyc-structure} below how $\TC(A;p)$ is defined in terms of B\"okstedt's model for $\THH$ in earlier literature on $\TC$. 

\begin{remark}
Classically, topological cyclic homology was defined under certain connectivity assumptions. In the language of \cite{LRRV17} a ring spectrum $A$ is called \textit{strictly connective} if $A_n$ is $(n-1)$-connected, \textit{very well-pointed} if $A_n$ is well-pointed and the unit $S^0 \to A_0$ is an $h$-cofibration, and \textit{convergent} if the structure map $A_n \to \Omega A_{n+1}$ is $n + \lambda(n)$-connected, where $\lambda(n) \to \infty$ when $n \to \infty$. If $A$ is strictly connective, very well-pointed and convergent, then $\THH(A)$ is a fibrant cyclotomic spectrum. In this case $\TC(\THH(A);p)$ is then equivalent to the original definition of $\TC(A;p)$.

We further observe that since in our setting $A$ is an orthogonal ring spectrum, the same conclusion holds if we just assume that $A$ is strictly connective and very well-pointed. This is because convergence is used only to invoke B\"{o}kstedt's approximation Lemma, which holds automatically if $A$ is semistable (see \cite[Proposition 2.4]{HM97}).
\end{remark}

We can now prove our second main theorem from the introduction. 
\begin{proof}[Proof of Theorem \ref{thm:TC-comparison-intro}]
Since a fibrant replacement in pre-cyclotomic spectra sends a stable equivalence between cyclotomic spectra to a stable equivalence between $C_{p^n}$-equivariant $\Omega$-spectra for all $n\geq 0$, it follows from Theorem~\ref{thm:THH-model-comparison} that the first two maps in the zig-zag
\[\TC((B^{\mathrm{cy}}(A))^{\mathrm{fib}};p) \xrightarrow{\sim}  \TC((\mathrm{thh}(\sh{}A))^{\mathrm{fib}};p) \xleftarrow{\sim} \TC((\THH(A))^{\mathrm{fib}};p) \ot \TC(A;p) \]
are stable equivalences. Now assume that $A$ is strictly connective. By \cite[Proposition 2.4]{HM97}, $\THH(A) = \mathrm{thh}(\Sigma^{\infty}A)$ is a $C_{p^n}$-$\Omega$-spectrum for all $n\geq 0$. Hence the fibrant replacement $\THH(A) \to \THH(A)^{\mathrm{fib}}$ in pre-cyclotomic spectra induces a stable equivalence on $\TC$. The final step is to check that the resulting spectrum $\TC(\THH(A);p)$ is naturally equivalent to the classical definition of $\TC(A;p)$. The only non-obvious part is that our restriction maps agree with the classical ones found in e.g. \cite{HM97}, and we do this in Proposition~\ref{prop:compatible-R-maps} below.

We now explain how we can bootstrap the equivalences for the $p$-typical $\TC$'s to equivalences of integral topological cyclic homology. We recall from \cite[\S 6.4.3]{DGM13} that for a cyclotomic spectrum $X$, Goodwillie's integral topological cyclic homology $\TC(X)$ is defined as the homotopy pullback of the diagram 
\[\xymatrix{X^{hS^1} \ar[r] &  \prod_p (\holim_{F}X^{hC_{p^n}})^{\sm}_p & \ar[l] \prod_p \TC(X;p)^{\sm}_p} \]
where $p$ varies through all primes. The homotopy limit is taken over the fixed-points inclusions induced by $C_{p^n}\subset C_{p^{n+1}}$ and it is equivalent to $X^{hS^1}$ after $p$-completion (see \cite[Lemma 6.3.1.1]{DGM13}). The right pointing map is induced by the fixed-points inclusions of $C_{p^n}\subset S^1$ and by the maps into the $p$-completions. The left pointing map is the product of the completions of the maps
\[
\TC(X;p)=\holim_{F,R} X^{C_{p^n}}\longrightarrow \holim_{F} X^{C_{p^n}}\longrightarrow \holim_{F} X^{hC_{p^n}}
\]
where the first map restricts the limit to the subcategory generated by the maps $F$ and the second map is the limit of the forgetful maps from fixed-points to homotopy fixed-points.
 Since this pullback diagram is natural with respect to morphisms of cyclotomic spectra the zig-zag above extends to a zig-zag on the homotopy pullbacks
\[\TC((B^{\mathrm{cy}}(A))^{\mathrm{fib}}) \xrightarrow{\sim}  \TC((\mathrm{thh}(\sh{}A))^{\mathrm{fib}}) \xleftarrow{\sim} \TC((\THH(A))^{\mathrm{fib}}) \ot \TC(A), \]
where the third map is an equivalence when $A$ is strictly connective.
\end{proof}

\subsection{Equivalence of restriction maps}\label{subsec:classic-cyc-structure}

It remains to check that our method for defining restriction maps gives the same result as the more classical method. For this purpose we drop the prime $p$ and notice that the definition of the restriction maps in~\eqref{eq:R-maps-via-geometric-fixed-points} works for any cyclic group $C_r$. Applying this to $\THH(A) = \mathrm{thh}(\Sigma^{\infty}A)$ with the cyclotomic structure maps from~\eqref{eq:restriction-on-thhEq} provides maps 
\begin{equation}\label{eq:R-maps-on-THH-from-geom-fixed-points} R_{r,s}\colon \rho_{rs}^*\mathrm{thh}(\Sigma^{\infty}A)^{C_{rs}} \to \rho_{s}^*\mathrm{thh}(\Sigma^{\infty}A)^{C_s}.
\end{equation}
To compare these to the analogous maps from \cite{HM97}, we recall the construction from \cite{HM97}; see also \cite{DGM13} and \cite{LRRV17}.  Consider for each $n$ the map of cyclic spaces which in simplicial level $q$ is defined as
\begin{equation}\label{eq:HM-cyclotomic-structure}
\begin{split}
&\; {(\mathrm{sd}_r \mathrm{thh}_{\bullet}(\Sigma^\infty A))}_q(\mathbb{R}^n)^{C_r}
\xrightarrow{=} (\hocolimsubscr_{\cI^{\times (q+1)r}}\, \Omega^{\cI}_{(q+1)r}(\Sigma^\infty A)(\mathbb{R}^n))^{C_r}
\\
\xrightarrow{\cong}&\; \hocolimsubscr_{\cI^{\times (q+1)}}\, ((\Omega^{\cI}_{(q+1)r}(\Sigma^\infty A) \circ \Delta_r)(\mathbb{R}^n))^{C_r} \\  
\xrightarrow{=}&\; \hocolimsubscr_{(\bld{n_0},\dots, \bld{n_q}) \in \cI^{\times (q+1)}} (\Omega^{r(n_0+ \dots + n_q)}\left((A_{n_0}\sm \dots \sm A_{n_q} )^{\sm r} \wedge S^n \right))^{C_r}\\ 
\overset{\chi}\to &\; \hocolimsubscr_{(\bld{n_0},\dots, \bld{n_q}) \in \cI^{\times (q+1)}} \Omega^{n_0+ \dots + n_q}\left(((A_{n_0}\sm \dots \sm A_{n_q})^{\sm r})^{C_r} \wedge S^n \right)\\ 
\xrightarrow{\cong}&\; \hocolimsubscr_{(\bld{n_0},\dots, \bld{n_q}) \in \cI^{\times (q+1)}} \Omega^{n_0+ \dots + n_q}(A_{n_0}\sm \dots \sm A_{n_q} \wedge S^n) \\  
\xrightarrow{=}&\;\mathrm{thh}_{q}(\Sigma^\infty A)(\mathbb{R}^n)\ .
\end{split}
\end{equation}
The unlabeled maps are the canonical commutation of fixed points with homotopy colimits and the unique natural isomorphism of spaces $(B^{\sma r})^{C_r} \cong B$. The map $\chi \colon (\Omega^V X)^G \to \Omega^{V^G} X^G$ restricts each equivariant map out of a sphere to the fixed points of that sphere. Of course, as $n$ varies these fit together into a map of orthogonal spectra.

\begin{remark} \label{notcyclic} It is essential that the sphere $S^n$ in~\eqref{eq:HM-cyclotomic-structure} is a trivial $S^1$-representation sphere. The analogous map 
\[{(\mathrm{sd}_r \mathrm{thh}_{\bullet}(\Sigma^\infty A))}_q(V)^{C_r} \to \mathrm{thh}_{q}(\Sigma^\infty A)(V^{C_r})\]
considered in \cite[2.5]{HM97} for a nontrivial $S^1$-representation sphere $S^V$ is simplicial. However, the obvious residual actions of $C_{q+1} \cong C_{(q+1)r}/C_r$ on 
\begin{align*}&{(\mathrm{sd}_r \mathrm{thh}_{\bullet}(\Sigma^\infty A))}_q(V)^{C_r}  \\  \cong& \hocolimsubscr_{(\bld{n_0},\dots, \bld{n_q}) \in \cI^{\times (q+1)}} (\Omega^{r(n_0+ \dots + n_q)}\left((A_{n_0}\sm \dots \sm A_{n_q} )^{\sm r} \wedge S^V \right))^{C_r}\end{align*}
for all $q\geq 0$ do not define a cyclic object. For example the cyclic identities $d_0t_q=d_q$ fail in general since the left hand side involves a potentially non-trivial action on $S^V$ and the right hand side does not. However, one can still show that the map is $S^1$-equivariant after taking geometric realizations with respect to the diagonal $S^1$-action on the source.  This follows by checking that the realized map is $C$-equivariant for any finite $C \leq S^1$ using subdivisions and then observing that $\mathbb{Q}/\mathbb{Z}$ is dense in $S^1$. We thank Amalie H\o genhaven for pointing out this argument to us.
\end{remark}

After geometric realization, the composite in~\eqref{eq:HM-cyclotomic-structure} provides an $S^1$-equivariant map of orthogonal spectra on the trivial universe,
\[\rho_r^*\mathrm{thh}(\Sigma^{\infty}A)^{C_r} \to \mathrm{thh}(\Sigma^{\infty}A).  \]
By further taking $C_s$-fixed points we also get an $S^1$-map 
\begin{equation}\label{eq:R-maps-on-THH-from-fixed-points} \overline{R}_{r,s}\colon  \rho_{rs}^*\mathrm{thh}(\Sigma^{\infty}A)^{C_{rs}} \to \rho_{s}^*\mathrm{thh}(\Sigma^{\infty}A)^{C_s}.\end{equation}

\begin{proposition}\label{prop:compatible-R-maps} The restriction maps $\rho_{rs}^*\mathrm{thh}(\Sigma^{\infty}A)^{C_{rs}} \to \rho_{s}^*\mathrm{thh}(\Sigma^{\infty}A)^{C_s}$ defined in \eqref{eq:R-maps-on-THH-from-geom-fixed-points} and \eqref{eq:R-maps-on-THH-from-fixed-points} coincide. 
\end{proposition}

\begin{proof}
It is enough to show that $R_{r,1}=\overline{R}_{r,1}$. For this we recall the canonical map $R$ from fixed points to geometric fixed points from \cite[Section V.4]{MM02} in more detail. For any orthogonal $C_r$-spectrum $X$, let $X^{C_r}$ denote the point-set level categorical fixed points. One has a natural morphism of orthogonal spectra 
\[ R\colon X^{C_r} \to  \Phi^{C_r}X=\int^{V} \mathcal{L}_S(V^{C_r}, -) \wedge X(V)^{C_r}, \]
where $\int^V$ is shorthand for the coequalizer from~\eqref{geometric_fixed_points_definition}.
The map $R$ sends $x \in X(\mathbb{R}^n)^{C_r}$ to the equivalence class of $\mathrm{id} \wedge x \in \mathcal{L}_S(\mathbb{R}^n, \mathbb{R}^n) \wedge X(\mathbb{R}^n)^{C_r}$.

In order to prove $R_{r,1}=\overline{R}_{r,1}$ it suffices to show that the composite of $R$ with the cyclotomic structure map of this paper
\[{(\mathrm{sd}_r \mathrm{thh}_{\bullet}(\Sigma^\infty A))}_q(\mathbb{R}^n)^{C_r} \overset{R}\to \Phi^{C_r} {(\mathrm{sd}_r \mathrm{thh}_{\bullet}(\Sigma^{\infty}A))}_q(\mathbb{R}^n) \to \mathrm{thh}_{q}(\Sigma^\infty A)(\mathbb{R}^n)\]
coincides with the cyclic map~\eqref{eq:HM-cyclotomic-structure}. This follows from several general observations about $R$. The first observation is that it commutes with the interchange of fixed points with homotopy colimits in the following way: for each $C_r$-diagram $Z$ over a $C_r$-category $\mathcal C$ the diagram
\[ \xymatrix @C=5em{
\hocolimsubscr_{\mathcal C^{C_r}}\, (Z\circ\Delta)^{C_r} \ar[r]^-{\hocolim\, R} \ar[d]^-\cong & \hocolimsubscr_{\mathcal C^{C_r}}\, \Phi^{C_r} (Z\circ\Delta)  \ar[d]^-\cong \\ (\hocolimsubscr_{\mathcal C}\, Z)^{C_r} \ar[r]^-R
 & \Phi^{C_r} \hocolimsubscr_{\mathcal C}\, Z}\]
 commutes, where $\Delta\colon\mathcal{C}^{C_r}\to\mathcal{C}$ is the inclusion of fixed-points.

The second observation is that on each spectrum level $n$ we have a commutative diagram 
\[ \xymatrix{{(\Omega^V X)(\mathbb{R}^n)}^{C_r} \ar[d]^-\chi \ar[r]^-R & \Phi^{C_r}(\Omega^V X)(\mathbb{R}^n) \ar[d]^-\lambda \\ \Omega^{V^{C_r}} (X(\mathbb{R}^n)^{C_r}) \ar[r]^-R & \Omega^{V^{C_r}} \Phi^{C_r}X(\mathbb{R}^n) }\]
where $\chi$ restricts an equivariant map to the fixed-points. This can be verified directly, by describing $\lambda$ as the map induced by the composite
\[
\bigvee_{W} \mathcal L_S\sm (\Omega^V X(W))^N\stackrel{\chi}{\longrightarrow} \bigvee_{W} \mathcal L_S \sm \Omega^{V^{N}} (X(W)^N)\longrightarrow \Omega^{V^{N}}\bigvee_{W} \mathcal L_S\sm  (X(W)^N)
\]
where $\mathcal L_S=\mathcal L_S(W^N,-)$ and the second map is the composite of the assembly of the loop functor and the canonical map that commutes a functor with a coproduct.
The last observation is that $R$ coincides with the diagonal isomorphism on the spectrum $\Sigma^{\infty}(B^{\wedge r})$ for any cofibrant space $B$, for instance using \cite[Theorem~1.2]{Mal17}.
\end{proof}

\begin{remark} The remaining steps that match the maps $\overline{R}_{r,s}$ with those in classical models of $\TC$ are straightforward translations between different models of spectra. For instance, in~\cite{HM97} Hesselholt and Madsen assume $A$ is a continuous functor instead of an orthogonal spectrum. But since they only use the symmetric spectrum structure of $A$, this poses no issue. Furthermore, they take the underlying $S^1$-prespectrum $t$ (indexed by $\mathcal U$) of the orthogonal $S^1$-spectrum $\THH(A)$ we defined here, and apply a thickening $t'$ and spectrification functor to get an $S^1$-equivariant spectrum $T$, in the sense of \cite{LMS}. But they prove that the spectrification map induces an equivalence on the $C$-fixed points of every representation level (\cite[Proposition 2.4]{HM97}, see also \cite{LRRV17}), and therefore the homotopy limit that defines $\TC$ takes this spectrification map to an equivalence of prespectra. Translating their definition of $\overline{R}$ back to the thickening $t'$ by restricting to the first term of the colimit systems for $T^{C_r}$ and $\Phi^{C_r} T$, translating further to $t$, and finally restricting to the trivial representation levels (since those are the only levels we use to define $\TC$), we get the map $t^{C_r}(\R^n) \to t((\R^n)^{C_r}) = t(\R^n)$ given by \eqref{eq:HM-cyclotomic-structure}.
In summary, the notion of $\TC$ defined using the maps $\overline{R}_{r,s}$ above is naturally equivalent to the one considered in \cite{HM97}.
\end{remark}

%\nocite{*}

\end{document}